\documentclass[11pt]{article}
\pdfoutput=1

\usepackage[latin9]{inputenc}
\usepackage[a4paper]{geometry}
\usepackage{microtype}
\usepackage{booktabs}
\usepackage{multirow}
\usepackage{adjustbox}
\usepackage{amsmath}
\usepackage{amsthm}
\usepackage{amssymb}
\usepackage{xfrac}
\PassOptionsToPackage{normalem}{ulem}
\usepackage{ulem}
\usepackage[unicode=true,pdfusetitle,
 bookmarks=true,bookmarksnumbered=false,bookmarksopen=false,
 breaklinks=false,pdfborder={0 0 0},pdfborderstyle={},backref=false,colorlinks=true]
 {hyperref}

\usepackage{titletoc}

    \titlecontents{chapter}
    [0em] %
    {\medskip\bfseries}
    {\thecontentslabel\quad}
    {}
    {\hfill\contentspage}

    \titlecontents{section}[0em]{\smallskip}%
    {\thecontentslabel\ \ \enspace}
    {}
    {\titlerule*[1pc]{.}\contentspage}%

 \usepackage[dvipsnames]{xcolor}
\hypersetup{colorlinks=true, linkcolor=RoyalBlue, citecolor=JungleGreen}

\makeatletter

\providecolor{lyxadded}{rgb}{0,0,1}
\providecolor{lyxdeleted}{rgb}{1,0,0}
\DeclareRobustCommand{\mklyxadded}[1]{\textcolor{lyxadded}\bgroup#1\egroup}
\DeclareRobustCommand{\mklyxdeleted}[1]{\textcolor{lyxdeleted}\bgroup\mklyxsout{#1}\egroup}
\DeclareRobustCommand{\mklyxsout}[1]{\ifx\\#1\else\sout{#1}\fi}

\numberwithin{equation}{section}
\numberwithin{figure}{section}
\numberwithin{table}{section}
\theoremstyle{plain}
\newtheorem{thm}{\protect\theoremname}[section]
\theoremstyle{plain}
\newtheorem*{thm*}{\protect\theoremname}
\theoremstyle{definition}
\newtheorem{defn}[thm]{\protect\definitionname}
\theoremstyle{plain}
\newtheorem{lem}[thm]{\protect\lemmaname}
\theoremstyle{plain}
\newtheorem{cor}[thm]{\protect\corollaryname}
\theoremstyle{remark}
\newtheorem{rem}[thm]{\protect\remarkname}
\theoremstyle{definition}
\newtheorem{example}[thm]{\protect\examplename}
\theoremstyle{plain}
\newtheorem{prop}[thm]{\protect\propositionname}

\newtheorem{quest}[thm]{Question}

\usepackage{graphicx}
\usepackage{setspace}
\usepackage{url}
\usepackage[perpage,symbol]{footmisc}
\usepackage{multicol}
\usepackage{etoolbox}
\usepackage{bbm}
\usepackage{enumitem}
\usepackage{caption}
\usepackage[capbesideposition={left,center}]{floatrow}
\usepackage{etoolbox}
\usepackage{fancyhdr}
\usepackage[abbrev,alphabetic]{amsrefs}
\setlist[enumerate]{itemsep=2pt, topsep=3pt, parsep=0pt}
\date{}

\DefineFNsymbols*{lamportnostar}[math]{{(\dagger)}{(\ddagger)}{(\S)}{(\P)}{(\|)}{(\dagger\dagger)}{(\ddagger\ddagger)}}
\setfnsymbol{lamportnostar}
\usepackage{mathrsfs, mathtools}

\newcounter{prcounter}

\makeatletter \let\saved@reset@nth@property\reset@nth@property \def\reset@nth@property#1#2#{\saved@reset@nth@property{#1}{#2}} \makeatother

\date{\today}

\newcommand{\lf}{\left}
\newcommand{\ri}{\right}
\newcommand{\md}{\middle}

\newcommand{\wh}{\widehat}

\DeclareMathOperator{\Aut}{Aut}

\DeclareMathOperator{\Stab}{Stab}

\DeclareMathOperator{\diag}{diag}
\DeclareMathOperator{\ad}{ad}

\DeclareMathOperator{\val}{val}

\newcommand{\m}[1]{\mathbf{#1}}
\newcommand{\mf}[1]{\mathfrak{#1}}
\newcommand{\mc}[1]{\mathcal{#1}}

\newcommand{\wtd}[1]{\widetilde{#1}}

\newcommand{\Sp}{\mathrm{Sp}}

\newcommand{\SO}{\mathrm{SO}}

\newcommand{\GL}{\mathrm{GL}}

\newcommand{\Frob}{\mathrm{Frob}}

\newcommand{\eps}{\epsilon}

\newcommand{\adj}{\mathrm{ad}}

\newcommand{\spl}{\mathrm{spl}}

\newcommand{\der}{\mathrm{der}}
\newcommand{\smc}{\mathrm{sc}}

\DeclareMathOperator{\disc}{disc}

\DeclareMathOperator{\Ram}{Ram}
\DeclareMathOperator{\dist}{dist}

\newcommand{\one}{\mathbbm{1}}

\theoremstyle{plain}

\theoremstyle{definition}
\theoremstyle{definition}

\newtheorem{rem}[thm]{Remark}

\newtheorem*{rems*}{Remarks}
\newtheorem*{disc*}{Discussion}
\setlist[enumerate,2]{leftmargin=0.5cm}

\allowdisplaybreaks[2]

\newcommand{\p}{\mathfrak{p}}
\newcommand{\z}{\zeta}
\newcommand{\n}{\natural}

\newcommand{\mB}{\mathcal{B}}
\newcommand{\mA}{\mathcal{A}}
\newcommand{\mC}{\mathcal{C}}
\newcommand{\mO}{\mathcal{O}}

\newcommand{\mI}{\mathcal{I}}
\newcommand{\mN}{\mathcal{N}}
\newcommand{\mZ}{\mathcal{Z}}

\newcommand{\mW}{\mathcal{W}}

\newcommand{\bfG}{\mathbf{G}}
\newcommand{\bfH}{\mathbf{H}}
\newcommand{\bfT}{\mathbf{T}}
\newcommand{\bfM}{\mathbf{M}}
\newcommand{\bfS}{\mathbf{S}}
\newcommand{\bfZ}{\mathbf{Z}}

\newcommand{\bfN}{\mathbf{N}}
\newcommand{\bfU}{\mathbf{U}}
\newcommand{\bfP}{\mathbf{P}}
\newcommand{\bfPU}{\mathbf{PU}}
\newcommand{\bfSU}{\mathbf{SU}}
\newcommand{\bfGU}{\mathbf{GU}}

\newcommand{\R}{\mathbb R}
\newcommand{\Q}{\mathbb Q}
\newcommand{\N}{\mathbb N}
\newcommand{\Z}{\mathbb Z}

\newcommand{\F}{\mathbb F}
\newcommand{\A}{\mathbb A}

\newcommand{\bF}{\mathbb{F}}
\newcommand{\bE}{\mathbb{E}}
\newcommand{\bK}{\mathbb{K}}
\newcommand{\bL}{\mathbb{L}}
\newcommand{\bG}{\mathbb{G}}

\AtBeginDocument{%
   \def\MR#1{}
}

\makeatother

\providecommand{\conjecturename}{Conjecture}
\providecommand{\corollaryname}{Corollary}
\providecommand{\definitionname}{Definition}
\providecommand{\examplename}{Example}
\providecommand{\lemmaname}{Lemma}
\providecommand{\propositionname}{Proposition}
\providecommand{\remarkname}{Remark}
\providecommand{\theoremname}{Theorem}

\begin{document}

\title{Optimal expansion in unitary groups I:\\ Golden adelic groups and cutoff}
\author{Rahul Dalal, Shai Evra, and Ori Parzanchevski} 
\maketitle

\begin{abstract}
In this paper and its companion \cite{DEP2} we construct and study special generating sets for $p$-arithmetic unitary Lie groups, extending the seminal works of Lubotzky, Phillips and Sarnak \cite{LPS86,LPS87,LPS88} to higher dimensions. 
The real embeddings of our sets form optimal topological generators for compact unitary Lie groups, and their $p$-adic embeddings give rise to Cayley graphs and complexes over finite unitary groups, which exhibit the cutoff phenomenon. 

In the current paper we prove two main results:
First, we construct $p$-arithmetic lattices in $\bfPU_4$ and $\bfPU_8$ which act transitively on their corresponding Bruhat-Tits buildings. These groups are used to construct Golden Gate sets on two and three qubits, which have optimal circuit length and efficient compiling algorithms.

Second, we study the cutoff phenomenon for quotients of the Bruhat-Tits building of a general semisimple $p$-adic group. We show that cutoff occurs for quotients which satisfy either the Ramanujan property or the weaker Sarnak-Xue Density Hypothesis (SXDH). For this purpose we bound the operator norms of a large class of Hecke operators in terms of the rate of decay of unitary representations. In the companion paper \cite{DEP2}, we prove such a SXDH result for unitary groups.


\end{abstract}

\tableofcontents

\section{Introduction} 

The main purpose of the present paper is twofold: 
\begin{enumerate}
\item Construct new examples of $p$-arithmetic projective unitary groups of class number one, or equivalently, which act transitively on a p-adic orbit in their corresponding Bruhat-Tits buildings.
\item Prove that the Sarnak-Xue Density Hypothesis (SXDH) in the level aspect implies the cutoff phenomenon for appropriate walks on quotients of Bruhat-Tits buildings of semisimple p-adic groups.
\end{enumerate}
In the second paper in this series \cite{DEP2}, we:
\begin{enumerate}[resume]
\item Prove the SXDH for definite algebraic unitary groups in dimensions 4 and 8, in both the weight and the level aspects.
\item Prove that SXDH in the weight aspect implies an almost-optimal covering property for the associated real Lie groups.
\item Use these results together with the ones from the current paper to construct Golden Gate sets on two and three qubits.
\end{enumerate}

To prove class number one for $p$-arithmetic groups we develop two different techniques: The first relies on Siegel-Prasad's volume formula (Proposition \ref{prop:U-class}) (which was also used in \cite{MSG12} and \cite{EP22,EFMP23}). The second makes use of the strong approximation theorem and a ``one-for-all'' philosophy (Corollary \ref{cor:U-one-for-all}; see also \cite{Lub99}, for a result of a similar flavor). The following theorem is a representative case of the various constructions presented in Section \ref{sec:G}:

\begin{thm} \label{thm:I-gold-Eis} 
Define the following Hermitian positive definite $4\times4$ matrix, 
\[
H =\left(\begin{array}{cc}
I_{2} & A\\
-A & I_{2}
\end{array}\right),\quad A=\tfrac{\sqrt{-3}}{3}\cdot\left(\begin{array}{cc}
1 & 1\\
1 & -1
\end{array}\right),
\]
For any prime $p \neq 2$, denote
\[
S_{p}:=\left\{ g\in M_4\left(\Z\left[\tfrac{1-\sqrt{-3}}{2}\right]\right)\;\md|\;{g^{*}Hg=p'\cdot H,\; g\equiv I_{4}\mod 2\atop g\text{ is not a scalar matrix}}\right\} 
\]
\[
\mbox{where}\qquad p'=\begin{cases}
\begin{array}{c}
p\\
p^{2}
\end{array} & \begin{array}{c}
p\equiv 0,1 \mod 3\\
p\equiv2 \mod 3
\end{array}\end{cases}.
\]
Then $\Lambda_p = \langle S_p \rangle \leq \bfPU_4^{\Q[\sqrt{-3
}],H}(\Q_p)$, is a congruence $p$-arithmetic subgroup of class number one. Furthermore, if $p \equiv 1 \mod 3$, then $\Lambda_p$ acts simply transitively on the vertices of the Bruhat-Tits building of $\bfPU_4^{\Q[\sqrt{-3
}],H}(\Q_p) \cong PGL_4(\Q_p)$.
\end{thm}

In \cite{MSG12}, the authors used Prasad's volume formula to prove a fantastic classification result for class number one $p$-arithmetic subgroups in $\bfPU_n$ for all $n\geq5$. They classify all such groups for $5 \leq n \leq 8$, and show that no such group exists for $n>8$. For applications related to quantum computation (see \cite{DEP2}), the case of  $\bfPU_4$ is of special interest, since it corresponds to computations in two qubits, which are known to be universal \cite{divincenzo1995two}. We remark that several of the examples which we construct build upon realizations of the $E_8$ lattice over the quadratic imaginary fields $\Q[\sqrt{-d}]$, for $d \in \{1,2,3,7\}$, following \cite{BF99,Wil12}. Other interesting examples are Propositions \ref{prop:G-nqs-Eis} and \ref{prop:G-nqs-Gauss}, in which the associated Bruhat-Tits building (of $\bfPU_4(\Q_p))$, for $p=2,3$) is a regular or biregular tree.

\medskip{}
Our second main goal in this paper is the cutoff phenomenon (see Definition \ref{defn:I-cutoff} below). We study it in very broad setting: let $\widetilde{X}$ be some infinite combinatorial object and let $G\leq\mathrm{Aut}(\widetilde{X})$. Let $\left\{ T_{\ell}\right\} $ be a sequence of $G$-equivariant (a.k.a.\ \textit{geometric}) operators defined on $\widetilde{X}$, or on a $G$-stable set of cells of $\widetilde{X}$, e.g.\ all cells of a certain dimension. Finally, let $\mathbf{u}$ be a finite dimensional $G$-invariant subspace of $L^{\infty}(\widetilde{X})$. We are concerned with the following general problem:
\begin{quest}\label{que:cutoff}
When does a family of finite quotients of $\widetilde{X}$ exhibit total-variation cutoff with respect to the sequence $\{T_{\ell}\}$, with $\mathbf{u}$ as the space of allowed stationary distributions?
\end{quest}

For example, if $\widetilde{X}$ is a regular tree, we can take $T_{\ell}$ to be the $\ell$-th power of the adjacency operator, or the $\ell$-th nonbacktracking operator, which takes a vertex in $\widetilde{X}$ to all vertices of graph distance $\ell$ from it (this is $K_t$ in \cite{NS23}). For $\mathbf{u}$ we can use the constant functions, thus allowing only the uniform stationary distribution, but we can also take $\mathbf{u}$ to be the space of functions which are constant on even and odd levels in the tree, which allows us to discuss cutoff on bipartite graphs as well. Various cutoff results on graphs of bounded degree were achieved in \cite{LS10,LP16,Sar19,NS23,GK22,EFMP23}, and in higher dimensions in \cite{LLP20,CP22} (the case of unbounded degree is markedly different, see the surveys \cite{Diaconis1996cutoffphenomenonfinite,salez2025modern}).

We now give the formal definitions. For a finite quotient $X=\Gamma\backslash\widetilde{X}$ ($\Gamma\leq G$) we denote by $T_{\ell}|_{X}$ the operator induced by $T_{\ell}$ on $X$, by $\mathbf{u}|_{X}$ the functions on $X$ whose lift to $\widetilde{X}$ is in $\mathbf{u}$, and by $\mathbf{u}|_{X}^{1}$ the space of probability distributions in $\mathbf{u}|_{X}$.

\begin{defn}[Cutoff] \label{defn:I-cutoff}
The \emph{$\varepsilon$-mixing time} of $X$ w.r.t.\ the sequence $\{ T_{\ell}\} $ and the stationary space $\mathbf{u}$ is
\[
t_{\varepsilon}(X) = \min\{ \ell \,|\,\mathrm{dist}_{TV}((T_{\ell}|_X)(\one_{x_0}),\mathbf{u}|_X^1) < \varepsilon\quad\forall x_0\in X \}.
\]
A family $\mathcal{F}$ of quotients of $\widetilde{X}$ \emph{exhibits cutoff} w.r.t.\ $\{ T_{\ell}\}$ and $\mathbf{u}$, at time $t\colon\mathbb{N}\rightarrow\mathbb{N}$ and with cutoff window $w\colon\mathbb{N}\rightarrow\mathbb{N}$, if $w(n)=o\left(t(n)\right)$ and for any $\varepsilon>0$ there exists $N_{\varepsilon}$ such that for every $X\in \mathcal{F}$ with $n = |X| \geq N_{\varepsilon}$,
\[
 t_{1-\varepsilon}(X) \geq t(n) - w(n) \qquad \text{ and } \qquad t_{\varepsilon}(X)\leq  t(n) + w(n).
\]
\end{defn}

We present now a simplified version of Theorem \ref{thm:C-cutoff-ram}, which proves cutoff for Ramanujan complexes as defined in \cite{LLP20}, giving a far-reaching generalization of \cite{LP16,LLP20} to any family of operators with appropriate growth and transience (see Definition \ref{defn:C-growth-trans}).

\begin{thm*}[\ref{thm:C-cutoff-ram}, simplified version] \label{thm:I-cutoff-ram}
Let $\mB$ be the affine Bruhat-Tits building of a semisimple p-adic group. If $\{ T_{\ell}\}$ is any sequence of transient geometric operators of $d$-exponential growth on $\mB$, then any family of Ramanujan quotients of $\mB$ exhibits cutoff at time $t(n)=\log_{d}(n)$ w.r.t.\ $\{T_{\ell}\}$, with $\mathbf{u}$ being the type functions on $\mB$ (see Theorem \ref{thm:C-cutoff-ram} for the precise definitions as well as the cutoff window $w(n)$).
\end{thm*}

By the above Theorem we would get the cutoff phenomenon for congruence quotients of a Bruhat-Tits building, \textbf{if} we could prove the Na\"ive Ramanujan Conjecture (NRC) for the underlying algebraic group. The NRC holds for $\bfU_2$, by combining Deligne's proof of the Ramanujan-Petersson Conjecture \cite{Del74}, with the Jacquet-Langlands correspondence \cite{JL70}. This was used in \cite{LPS88} to construct Ramanujan Cayley graphs, in \cite{LPS86,LPS87} to prove an almost-optimal covering property for Hecke points on the sphere, and in \cite{PS18} to construct golden gate sets in $PU(2)$. The NRC also holds in some higher rank groups arising from division algebras \cite{HT01,lafforgue2002chtoucas, Clozel2002Automorphicformsand,badulescu2017global}, but it does not hold in general for the matrix unitary algebraic groups $\bfU_n$ with $n\geq 3$ -- see \cite{HPS79}, \cite[Theorem~1.4]{LSV05}, explicit examples of failure in \cite{EFMP23} and \cite[Proposition~8.5]{EGGG23}, and the surveys \cite{Sar05,Li19}).

While the NRC no longer holds in general, in \cite{EP22,EFMP23} the authors were able to construct specific families of congruence subgroups of $\bfU_3$ which do satisfy the NRC, which were used to construct Ramanujan complexes and bigraphs, as well as golden gate sets in $PU(3)$. However, we expect that for matrix $\bfU_n$, $n\geq 4$, an infinite family of Ramanujan congruence subgroups does not exist, and one is forced to address the failure of the NRC. We also note that all existing constructions of Ramanujan complexes arise from forms of $GL_n$ \cite{LSV05b,li2004ramanujan,sarveniazi2007explicit,first2016ramanujan,EP22,EFMP23,DMZ26}, and we conjecture that in other reductive groups the NRC must fail.

Being forced to force the failure of NRC, we proceed according to the Density strategy which was developed in \cite{Sar90,SX91,Sarnak2015Appendixto2015,sarnak2016families,PS18,GK23} of replacing the NRC with a variant of the Sarnak-Xue Density Hypothesis (SXDH) in the level aspect (Definition \ref{defn:C-SXDH-gap}). See the introduction in \cite{EGGG23} for more on this form of the SXDH.

\medskip{}

We introduce some notations: Let $G$ be a semisimple $p$-adic group, $\mB$ its Bruhat-Tits building and $\mI \leq G$ an Iwahori subgroup. For a lattice $\Gamma \leq G$, denote by $L^2(\Gamma \backslash G)$ its right quasiregular representation, and by $L^2_\sigma(\Gamma \backslash G)$ the subspace of subrepresentations whose rate of decay of matrix coefficients is precisely $\sigma$. In these settings, the NRC is equivalent to having $L^2_\sigma(\Gamma \backslash G) = 0$ for any congruence subgroup $\Gamma \leq G$ and any $2< \sigma <\infty$.

\begin{defn} \label{defn:I-SXDH}
We say that a family $\{ X_{n}=\Gamma_{n}\backslash\mB\} $ of quotients of $\mB$ satisfies the Sarnak-Xue Density Hypothesis (SXDH) if for any $2\leq\sigma<\infty$,
\begin{equation} \label{eq:SXDH-only}
\sum\nolimits_{\tau \geq \sigma}\dim L_\tau^2(\Gamma_n \backslash G)^\mI \ll (\log|X_n|)^{O(1)} |X_n|^{2/\sigma}.
\end{equation}
\end{defn}

We now state a simplified version of our more general criterion for cutoff for finite quotients of buildings, which satisfy the SXDH and a \textit{spectral gap} property (see Definition \ref{defn:C-SXDH-gap}), as well as an almost-transitive property (see the paragraph before Theorem \ref{thm:C-cutoff-SX}).

\begin{thm*}[\ref{thm:C-cutoff-SX}, simplified version] \label{thm:I-cutoff-SX}
Let $\{ T_{\ell}\}$ be any sequence of transient geometric operators of $d$-exponential growth, on the Bruhat-Tits building $\mB$ of a semisimple p-adic group. Let  $\{X_n\}$ be a family of finite quotients of $\mB$, which is almost-transitive and satisfies the SXDH and spectral gap properties. Then $\{X_n\}$ exhibits cutoff at time $t(|X_n|)=\log_{d}(|X_n|)$ w.r.t.\ $T_{\ell}$ and $\mathbf{u}$ (see Theorem \ref{thm:C-cutoff-SX} for the precise definitions and cutoff window).
\end{thm*}

In order to prove Theorems \ref{thm:C-cutoff-ram} and \ref{thm:C-cutoff-SX}, it is necessary to provide spectral upper bounds on the operator norms of general Hecke operators, which we expect to be useful for other applications as well. In particular, these bounds will be needed in \cite{DEP2} to prove an almost-optimal almost-covering property for the arithmetic gate sets constructed there.

\begin{thm*}[\ref{thm:H-General-bound}, simplified version] \label{thm:I-Hecke-bound}
Let $P \leq G$ be a compact open subgroup and $S \subset G$, $2 \leq |S| <\infty$ satisfying $PSP=SP$ and $| S^{-1}S \cap P | \ll 1$. Assume that either:
\begin{enumerate}
\item $P$ contains the Iwahori subgroup $\mI$ and $S\nsubseteq N_G(\mI)$, or
\item $P$ is a flat stabilizer, $S$ is a $P$-translation set and $S \nsubseteq P$ (see Section \ref{sec:H} for the precise definitions of flat stabilizer and translation set).
\end{enumerate}
If $\pi$ is a unitary irreducible representation of $G$, with rate of decay $\sigma = \sigma(\pi)$, then
\[
\big\|\one_{P S P}\mid_{\pi^P} \big\|_{op} \ll (\log |S|)^{O(1)} |S|^{\frac{\sigma - 1}{\sigma}}, 
\] 
\end{thm*}

\paragraph{Organization of the paper.}

In Sections \ref{sec:U} and \ref{sec:BT} we collect facts about arithmetic unitary groups, and on Bruhat-Tits buildings of semisimple groups, emphasizing the details particularly important to our applications. In Section \ref{sec:G}, we define the notions of golden adelic groups, and present new constructions in $\bfPU_4$. In Section \ref{sec:H} we study the action of Hecke operators on unitary representations and obtain the spectral bounds which are needed for our purposes. Finally, in Section \ref{sec:C} we prove cutoff mixing time results for geometric operators assuming either Ramanujan or the Sarnak-Xue Density Hypothesis. 


\paragraph{Acknowledgments.}
We would like to thank Yiannis Sakellaridis, Alexander Hazeltine, Alberto Minguez, Mathilde Gerbelli-Gauthier, Zachary Steir, Ian Texeira, and David Schwein and Jiandi Zou for many helpful conversations about Bruhat-Tits theory. 

RD was partially funded by Principal Investigator project PAT4832423 of the Austrian Science Fund (FWF) and NSF postdoctoral grant 2103149 while working on this project. SE was supported supported by ISF grant 1577/23. OP was supported by ISF grant 2990/21, and by the Institute for Advanced Study, through a Kovner Foundation membership.

\section{Unitary groups}\label{sec:U}

In this section, we collect some facts about algebraic and arithmetic unitary groups and their class number, and give two criteria for class number one.

Let $F$ be a totally real number field and $\mO$ its ring of integers. Let $V=V_f\sqcup V_\infty$ be the set of places of $F$, divided to finite and infinite ones. For any $v \in V$, let $F_v$ be the $v$-completion of $F$ and, if $\p \in V_{f}$, let $\mO_{\p}$ be the ring of integers of $F_{\p}$ and $q_{\p}$ the residue field degree of $F_{\p}$. For $S\subset V$, define the ring of $S$-integers of $F$ by $\mO[1/S] = F \cap\bigcap_{\p\in V_{f}\setminus S}\mO_{\p}$. When $S=\{ \p\} \bigcup V_{\infty}$, we abbreviate $\mO\left[1/\p\right] :=\mO\left[1/S\right]$, called the ring of $\p$-integers of $F$. 

Let $E$ be a CM extension of $F$ and $\mO_{E}$ its ring of integers. Let $3\leq n\in\N$ and let $H\in GL_n(E)$ be a totally positive-definite Hermitian matrix\footnote{In fact, all the claims in this section hold for $n=2$ as well, except over a place at which the group is non-quasisplit - see Lemma \ref{lem:U-points}.}. We can assume by scaling that $H\in M_{n}(\mO_E)$.
Denote by $\bfU_n^{E,H}$ the unitary group scheme over $\mO$ with respect to $E$ and $H$, defined for any $\mO$-algebra $A$ by 
\[
\bfU_n^{E,H}(A):=\{ g\in GL_n(A\otimes_{\mO}\mO_E) \,:\, g^* H g = H \} .
\]
Define the special unitary and projective unitary group schemes $\bfSU_n^{E,H}$ and $\bfPU_n^{E,H}$ by
\[
\bfSU_n^{E,H}:=\{ g\in \bfU_n^{E,H}\,:\,\det(g)=1\} ,\qquad \bfPU_n^{E,H}:=\bfU_n^{E,H}/\bfU_1^E,
\]
where $\bfU_1^E(A):=\{ x\in (A\otimes_{\mO}\mO_E)^\times \,:\, \bar{x}x=1\}$ is identified with the scalar matrices of $\bfU_n^{E,H}$. Finally, define the similitude unitary group scheme by 
\[
\bfGU_n^{E,H} := \{ g \in GL_n(A\otimes_{\mO}\mO_E ) \,:\, g^{*} H g = \alpha H, \alpha \in (A\otimes_{\mO}\mO_{E})^\times\}.
\]

We have equalities on points: $\bfPU_n^{E,H}(\R) = \bfU_n^{E,H}(\R)/\bfU_1^E(\R)$ and $\bfPU_n^{E,H}(K) = \bfGU_n(K)/(K \otimes_\mO \mO_E)^\times$ for local and global fields $K/F$. However, this does not necessarily hold more generally.

The discriminant of $H$ is given by
\begin{equation}\label{eq:U-disc}
\disc(H) := (-1)^{n(n-1)/2} \det(H) \in F^\times/N_{E/F}(E^\times).
\end{equation}
Since $H$ is totally positive-definite, we recall a well-known fact (using the classification of Hermitian forms over $p$-adic fields from \cite{Jacobowitz1962Hermitianformsover}):

\begin{lem}\label{lem:U-points}
Let $v\in V$, $\bF = F_v$ and $G = \bfU^{E,H}_n(\bF)$. Then $G$ is isomorphic to
\begin{itemize}
\item  $U(n)$, the classical (i.e.\ complex-entry) compact unitary group in $n$ variables, if $v\in V_\infty$,
\item $GL_n(\bF)$, the general linear group, if $v = \p$ is finite and split (in $\mO_E$),
\item $U_n^{qs}(\bF)$, the unique quasisplit unitary group for the extension $\bE = \bF \otimes_F E/\bF$, if $v = \p$ is non-split finite and either $n$ is odd or $\disc(H) \in N_{\bE/\bF}(\bE^\times)$,
\item  $U_n^{nqs}(\bF)$, the unique non-quasisplit unitary group for the extension $\bE = \bF \otimes_F E/\bF$, if $v = \p$ is non-split finite, $n$ is even and $\disc(H) \not \in N_{\bE/\bF}(\bE^\times)$.
\end{itemize}
The semisimple rank in the four cases is $0, n-1, \left\lfloor n/2\right\rfloor, n/2-1$, respectively.
\end{lem}

Let $\bfG$ be either $\bfU_n^{E,H}$, $\bfSU_n^{E,H}$, $\bfPU_n^{E,H}$, or $\bfGU_n^{E,H}$. Let $G = \bfG(F)$, denote $G_\infty := \prod_{v\in V_{\infty}} G_v$, $G_v :=\bfG(F_v)$ (for any $v \in V$), and for any $\p\in V_f$ denote $K_\p := \bfG(\mO_\p)$ and $\Gamma_\p := \bfG(\mO[1/\p])$. Call $\Gamma_\p$ the principal $\p$-arithmetic subgroup of $G$, and call any finite index subgroup of $\Gamma_\p$ a $\p$-arithmetic subgroup. Call $\bfG(\mO)$ the principal arithmetic subgroup of $G$ and any finite index subgroup of $\bfG(\mO)$ arithmetic.

\begin{lem}\label{lem:U-BHC} 
Let $\bfG$ be $\bfU_n^{E,H}$, $\bfSU_n^{E,H}$ or $\bfPU_n^{E,H}$. Then:
\begin{enumerate}
\item Any arithmetic subgroup of $G$ is finite.
\item Any $\p$-arithmetic subgroup of $G$ is a cocompact lattice in $G_{\p}$.
\item For any $\p$-arithmetic subgroup $\Gamma$ of $G$, denote $G'_{\infty} = \{g\in G_\infty \,|\, \det(g) = 1\}$ and $\Gamma' = \Gamma \cap G'_\infty$. Then $\Gamma'$ is a dense subgroup of $G'_{\infty}$.
\end{enumerate}
\end{lem}

\begin{proof}
First note that, since $H$ is totally positive-definite, we get that for any $v\in V_{\infty}$, $G_v \cong U(n)$, $SU(n)$ or $PU(n)$, the unitary, special unitary, or projective unitary compact Lie groups; hence, $G_{\infty}$ is a compact Lie group. By Borel--Harish-Chandra theory \cite{borel1962arithmetic}, we get that any arithmetic subgroup of $G$ is a cocompact lattice of $G_{\infty}$ and any $\p$-arithmetic subgroup of $G$ is a cocompact lattice of $G_{\infty} \times G_{\p}$. 

Claim 1 follows from the fact that a discrete subgroup of a compact group is finite. Claim 2 follows from the fact that the projection of a cocompact lattice of $G_{\infty}\times G_{\p}$ onto the second component remains a cocompact lattice of $G_{\p}$ (since $G_{\infty}$ is compact). 
Claim 3 follows from the strong approximation theorem for $\bfSU_n$ (see \cite{Platonov1994Algebraicgroupsand}), and the fact that $G_{\p}$ is non-compact since $\mathrm{rank}_{F_{\p}}(G_{\p}) >0$.
\end{proof}

Denote the ring of adeles of $F$ by
\[
\A := \mathop{\sideset{}{'}\prod}\nolimits_{v\in V} F_v = \bigg\{x = (x_v)_v \in \prod\nolimits_{v\in V} F_v \;\bigg|\; x_v \in \mO_v \mbox{ for all but finitely many } v \in V_f \bigg\}.
\]
Let $\widehat{\mO} := \prod_{\p \in V_{f}}\mO_{\p}$, let $F_{\infty} = \prod_{v\in V_{\infty}}F_{v}$, and consider $F$ embedded diagonally in $\A$. Then $\A$ is a locally compact ring, $F$ is a discrete subring, $F_{\infty}\widehat{\mO}$ (embedded coordinate-wise) is an open subring, and $\A=F\cdot F_{\infty}\widehat{\mO}=F+F_{\infty}\widehat{\mO}$. 

\begin{lem}\label{lem:U-U1E} 
If $h(E)=1$ (the class number of $E$) then $\bfU_1^E (\A) = \bfU_1^E(F)\cdot \bfU_1^E(F_\infty\widehat{\mO})$.
\end{lem}

\begin{proof}
Let $x=(x_v) \in \bfU_1^E(\A)$, i.e.\ $x_v\in\bfU^E_1(F_v)$ for every $v$ and $x_\ell \in \bfU_1^E(\mO_\ell)$ for almost all finite places $\ell$. Note that if $\ell$ is non-split in $E$, then $\bfU_1^E(F_\ell) = \bfU_1^E(\mO_\ell)$ by valuation considerations. If $\ell=(\lambda)(\overline{\lambda})$ is split then  $\mO_E$ embeds in $\mO_{\ell}$ by a map which takes $\lambda$ to a uniformiser and $\overline{\lambda}$
to a unit. This map induces $\bfU^{E}_{1}(F_{\ell})\cong F^{\times}_{\ell}$
and $\bfU^{E}_{1}(\mathcal{O}_{\ell})\cong\mathcal{O}^{\times}_{\ell}$, so that $\lambda/\overline{\lambda}\in\bfU^{E}_{1}(F)$ is
a uniformiser in $\bfU^{E}_{1}(F_{\ell})\cong F^{\times}_{\ell}$,
and a unit in all other finite places. Applying this to finitely many $\ell$, we obtain that $x\in\bfU^{E}_{1}(F)\bfU^{E}_{1}(F_{\infty}\widehat{\mO})$.
\end{proof}

\begin{lem} \label{lem:U-SA} 
Let $\bfG$ be $\bfU_n^{E,H}$, $\bfSU_n^{E,H}$ or $\bfPU_n^{E,H}$, and if $\bfG\ne \bfU_n^{E,H}$ assume that $h(E)=1$.
Then for any prime $\p \in V_f$,
\[
\bfG(\A) = \bfG(F) \cdot G_{\infty} G_{\p} \prod\nolimits_{\p\ne\ell\in V_f}K_{\ell}.
\]
\end{lem}

\begin{proof}
If $\bfG = \bfSU_n^{E,H}$ then $\bfG(F)G_{\p}$ is dense in $\bfG(\A)$ by the strong approximation property, so the claim follows as $K=G_{\infty} \prod_{\ell\in V_f}K_{\ell}$ is open. If $\bfG = \bfU_n^{E,H}$, then by the claim for $\bfSU_n^{E,H}$, we get that $\bfG(F)\cdot G_{\infty}G_{\p}\prod_{\p\ne\ell\in V_{f}}K_{\ell}$ contains $\bfSU_n^{E,H}(\A)$, and by Lemma \ref{lem:U-U1E} it also contains $\bfU_1^E(\A)=\det(\bfG(\A))$, hence it must equal $\bfG(\A)$. The claim for $\bfPU_n$ follows.
\end{proof}

Let $\bfG$ be $\bfU_n^{E,H}$, $\bfSU_n^{E,H}$ or $\bfPU_n^{E,H}$.
By \cite{borel1962arithmetic}, $\bfG(F)$ is a cocompact lattice of $\bfG(\A)$. Since $K:=G_{\infty}\bfG(\widehat{\mO})$ is a compact open subgroup, the double quotient space $\bfG(F)\backslash \bfG(\A)/K$ is finite. 

\begin{lem}  \label{lem:U-class-p} 
Let $\bfG$ be $\bfU_n^{E,H}$, $\bfSU_n^{E,H}$ or $\bfPU_n^{E,H}$, and if $\bfG \ne \bfSU_n^{E,H}$ assume that $E$ is of class number one. Then for any prime $\p \in V_f$, the map
\[
G_\p \to \bfG(\A), \qquad g\mapsto\left(g_{v}\right)_{v\in V},\; g_{v} = \begin{cases} \begin{array}{c} g\\ 1 \end{array} & \begin{array}{c} v=\p\\ v\ne\p \end{array}\end{cases},
\]
induces a bijective map
\[
\Gamma_\p \backslash G_\p/K_\p \to \bfG(F)\backslash \bfG(\A)/K.
\] 
\end{lem}

\begin{proof}
The map is  well defined since $\Gamma_{\p} \leq \bfG(F)$ and $K_\p \leq K$. The map is obviously injective, so the claim boils down to proving it is surjective, which follows from the strong approximation type result of Lemma~\ref{lem:U-SA}. 
\end{proof}

\begin{cor} \label{cor:U-one-for-all}
Let $\bfG$ be $\bfU_n^{E,H}$, $\bfSU_n^{E,H}$ or $\bfPU_n^{E,H}$, and if $\bfG \ne \bfSU_n^{E,H}$ assume that $E$ is of class number one. The following are equivalent:
\begin{itemize}
\item $G_\p = \Gamma_\p K_\p$ for some prime $\p \in V_f$.
\item $G_\p = \Gamma_\p K_\p$ for all primes $\p \in V_f$.\hfill \qedsymbol
\end{itemize}
\end{cor}


\begin{defn}
Let $\bfG$ be $\bfU_n^{E,H}$, $\bfSU_n^{E,H}$ or $\bfPU_n^{E,H}$.
The class and $\p$-class numbers of $\bfG$, for any prime $\p \in V_f$, with respect to $K = G_\infty \prod_{\ell \in V_f} K_\ell$, are defined to be 
\[
c(\bfG) := |\bfG(F) \backslash \bfG(\A)/K|,\qquad c_\p(\bfG) := |\Gamma_\p\backslash G_\p/K_\p|.
\]
\end{defn}

Note that under the assumptions of Lemma \ref{lem:U-class-p} we have 
$c_\p(\bfG) = c(\bfG)$ for any prime $\p$.
When $\bfG$ satisfies the property that $K_\p = \bfG(\mO_\p)$ is a ``special maximal compact'' subgroup of $G_\p = \bfG(F_\p)$ for every $\p$ (see \S\ref{sec:BT})\footnote{Beware that \cite{gan2001exact} uses a non-standard definition of ``maximal parahoric'' that corresponds to the standard notion of special maximal compact. This is pointed out above (2.2) therein.}, the Mass formula of \cite{gan2001exact} provides a good test if $\bfG$ has class number one. 

Denote by $\mathrm{disc}(E)$ the discriminant of $E$ and by $\mathrm{disc}(H)$ that of $H$ as in \eqref{eq:U-disc}. Denote by $\mathrm{Ram}(E)$ the set of primes that divide $\mathrm{disc}(E)$ and by $\mathrm{Ram}(H)$ the sets of primes $\p$ at which $\disc(H) \notin N_{E_\p/F_\p}(E_\p^\times)$. Denote by $\z_F(s)$ and $L(s,\chi_{E/F})$ the Dedekind zeta function of $F$ and the Dirichlet $L$-function of $\chi_{E/F}$, the quadratic Dirichlet character associated to $E/F$ by class field theory. 

\begin{defn} \label{def:U-mass} 
Let $\bfG = \bfU_n^{E,H}$. Define the set of ramified primes of $\bfG$ to be 
\[
\mathrm{Ram}(\bfG) := \mathrm{Ram}(E) \cup \mathrm{Ram}(H),
\]
define the $\lambda$-constant of $\bfG$ to be  
$\lambda(\bfG) := \prod_{\ell\in\mathrm{Ram}(\bfG)}\lambda_{\ell}$ 
with
\[
\lambda_{\ell} = \begin{cases}
1/2 & 2\nmid n,\;\ell\in\mbox{Ram}(E)\\
(q_{\ell}^{n}+1)/(q_{\ell}+1) & 2\nmid n,\;\ell\not\in\mbox{Ram}(E)\\
1 & 2\mid n, \; \ell \not\in \mbox{Ram}(H)\\
(q_{\ell}^{n}-1)/(q_{\ell}+1) & 2\mid n,\; \ell\not\in\mbox{Ram}(E), \; \ell \in \mbox{Ram}(H)\\
(1/2)(q_{\ell}^{n}-1)/(q_{\ell}^
{n/2}+1) & 2\mid n,\; \ell\in\mbox{Ram}(E), \; \ell \in \mbox{Ram}(H), \; K_\ell \text{ extraspecial} \\
(1/2)(q_{\ell}^{n}-1)/(q_{\ell}+1) & 2\mid n,\;\ell\in\mbox{Ram}(E), \; \ell \in \mbox{Ram}(H), \; K_\ell \text{ not extraspecial.}
\end{cases}
\]
Define the $L$-special value of $\bfG$ to be 
\[
L(\bfG) := 2^{-n[F:\Q]+1} \cdot \prod_{r=1}^{\lfloor n/2\rfloor}\z_{F}(1-2r) \cdot \prod_{r=1}^{\lfloor (n+1)/2\rfloor} L(2-2r,\chi_{E/F}),
\]
and finally, define the mass constant of $\bfG$ to be $R(\bfG) := L(\bfG) \cdot \lambda(\bfG)$.
\end{defn}

The following consequence of the Mass formula in \cite{gan2001exact} gives us our first criterion for proving class number one.

\begin{prop} \label{prop:U-class} 
Let $\bfG = \bfU_n^{E,H}$ be such that $\bfG(\mO_\p)$ is a special maximal compact subgroup of $\bfG(F_\p)$, for all primes $\p \in V_f$. Then 
\[
c(\bfG) = 1 \qquad \Leftrightarrow \qquad |\bfG(\mO)|^{-1} = R(\bfG).
\]
\end{prop}

\begin{proof}
Let $\mu_{\mathrm{tam}}$ be the Tamagawa measure of $\bfG(\A)$ and denote the mass of $\bfG$ by
\[
\mathrm{Mass}(\bfG) := \frac{\mu_{\mathrm{tam}}(\bfG(F)\backslash \bfG(\A))}{\mu_{\mathrm{tam}}(K)}.
\]
On the one hand, Propositions 4.4 and 4.5 in \cite{gan2001exact},
yield
\[
\mathrm{Mass}(\bfG) = R(\bfG).
\]
Note that Tables 1 and 2 in \cite{gan2001exact} do not cover the cases of $\ell$ ramified, $K_\ell$ extraspecial, and $G_\ell$ either odd or even non-quasisplit. These cases can be filled in by \cite{gan2001exact}*{(2.12)}. For the first case, the reductive quotient of the special fiber of the corresponding parahoric integral model is $\SO_n$ since $G_\ell$ corresponds to type $C$-$BC_{(n-1)/2}$ in the tables of \cite{Tit79}. This has the same number of points as $\Sp_{n-1}$. In the second case, the reductive quotient is ${}^2\SO_n$ coming from type ${}^2\!B$-$C_{n/2}$. In both cases, the parahoric is of index two in the maximal compact using results from \cite{HR10} (see \S\ref{sec:BT} for more details) and formulas for point counts over finite fields which can be found in \cite{ATLAS}*{\S2}. 

On the other hand, by Siegel-Braun mass formula \cite{braun1941hermitian}, the mass of $G$ equals the sum over the representatives of the genus of $G$, weighted by the reciprocal of the size of the associated arithmetic finite group:
\[
\mathrm{Mass}(\bfG) = \sum_{g\in \bfG(F)\backslash  \bfG(\A)/K} |\bfG(F)\cap gKg^{-1}|^{-1} = |\bfG(\mO)|^{-1} + \sum_{1\ne g\in \bfG(F)\backslash \bfG(\A)/K}\cdots.
\]
Since each member in the sum is positive rational number, comparing both estimates shows that $|\bfG(\mO)|^{-1} = R(\bfG)$ if and only if $c(\bfG)=1$.
\end{proof}

\begin{rem}\label{rem:U-class-1}
The are only finitely many values of $n$, $\mathrm{disc}(E)$, and $\mathrm{disc}(H)$ such that the definite unitary group $\bfG= \bfU_n^{E,H}$ has class number one (see \cite{borel1989finiteness}). In fact, Mohammadi and Salehi-Golsefidy in \cite{MSG12} showed that $n=8$ is the threshold for definite unitary groups of class number one; namely, for any $n>8$, there are no class number one definite unitary groups of the form $\bfG = \bfU_n^{E,H}$ and for any $4 < n \leq 8$, there exists class number one definite unitary groups of the form $\bfG = \bfU_n^{E,H}$.
\end{rem}

Next we wish to describe our second criterion for proving class number one. First we record the following general Lemma.

\begin{lem} \label{lem:U-transitive-criterion}
Let $X$ be a connected $k$-regular graph and $v_0$ a vertex in $X$. Let $\Lambda \leq \mathrm{Aut}(X)$ be a discrete subgroup, and denote $C = \mathrm{Stab}_\Lambda(v_0)$ and $S = \{g\in \Lambda \,|\, \mathrm{dist}(g.v_0,v_0) = 1 \}$.
If $|S|/|C| = k$, then $\Lambda$ acts transitively on the vertices of $X$.
\end{lem}

\begin{proof} 
Denote by $N(v)$ the neighbors of a vertex $v\in X$. Note that $|S|/|C|$ is the number of neighbors of $v_0$ in the orbit $\Lambda.v_0$. By $|S|/|C| = k$ we get that $N(v_0) \subseteq \Lambda.v_0$, and translating by $\Lambda$ we obtain that $N(v) \subset \Lambda.v$ for any $v\in \Lambda.v$. It follows that the subgraph induced by $\Lambda.v$ is open and closed, so by connectedness it equals $X$.
\end{proof}

The above Lemma specialize to the following Corollary when we take $X$ to be the Bruhat-Tits building $\mB$ of the $p$-adic group $PGL_n(F_{\p})$. See Section \ref{sec:BT} and especially Example \ref{exa:BT-B(PGL_n(F))} for more details information on Bruhat-Tits buildings and an explicit description of the $\mB$. We note that the underlying graph of $\mB$ is $D_n(q_{\p})$-regular, where
\begin{equation}\label{eq:D_n(q)}
D_n(q_{\p}) := \sum_{k=1}^{n-1} {n\choose k}_{q_{\p}}
, \qquad  {n\choose k}_{q} := \prod_{i=0}^{k-1} \frac{q^{n-i} - 1}{q^{k-i} - 1}.    
\end{equation}

\begin{cor} \label{cor:U-transitive-criterion}
Let $G = PGL_n(F_{\p})$, $K = PGL_n({\mc O}_{\p})$, $\mB = \mB(PGL_n(F_{\p}))$, and $v_0$ the vertex in $\mB$ whose stabilizer is $K$.
For a discrete subgroup $\Lambda \leq G$, denote $C = \Lambda \cap K = \{g\in \Lambda \,|\, g.v_0 = v_0 \}$ and $S = \{g\in \Lambda \,|\, \mathrm{dist}(g.v_0,v_0) = 1 \}$.
If $|S|/|C| = D_n(q_{\p})$ then $G = \Lambda \cdot K$.
\end{cor}

\begin{proof}
By Lemma \ref{lem:U-transitive-criterion} we get that $\Lambda$ acts transitively on the vertices of the building $\mB$.
Since $G \leq \mathrm{Aut}(\mB)$ and $K = \mathrm{Stab}_G(v_0)$, we get that $G = \Lambda \cdot K$.
\end{proof}


We now give our second criterion for proving class number one (which avoids calculating the mass formula). See also Proposition \ref{prop:BT-criterion2-class1} below.

\begin{cor} \label{cor:U-transitive-criterion2}
Let $\bfG = \bfPU_{n}^{E,H}$, with $E$ of class number one, $\p$ a split prime, $G_\p = \bfG(F_{\p})$, $\Lambda_{\p}  = \bfG({\mc O}[1/\p]) \leq  G_\p$, and $v_0$ a vertex in $\mB(G_\p) = \mB(PGL_n(F_{\p}))$. Denote $C_{\p} = \{g \in \Lambda_{\p}\,|\, g.v_0 = v_0\}$ and $S_{\p} = \{g \in \Lambda_{\p}\,|\, \mathrm{dist}(g.v_0,v_0) = 1\}$. If $|S_{\p}|/|C_{\p}| = D_n(q_{\p})$, then $\bfG$ is of class number one.
\end{cor}

\begin{proof}
Follows from combining Corollary \ref{cor:U-transitive-criterion} with Lemma \ref{lem:U-class-p}.
\end{proof}

\section{Bruhat-Tits buildings}\label{sec:BT}

This section summaries facts about the theory of Bruhat--Tits
buildings of connected reductive groups over local non-archimedean fields, following \cite{KP23} (see also \cite{Tit79}). 

Let $\bF$ be a non-Archimedean local field, $\mO_{\bF}\subset\bF$ its ring of integers, $\varpi_{\bF}\in\mO_{\bF}$ a uniformizer, $q_{\bF} = \left| \mO_{\bF}/\varpi_{\bF} \mO_{\bF}\right|$ its residue degree and $\mathrm{val}\colon \bF^\times \to \Z$ the valuation map. For example,  $\bF = \Q_p$, $\mO_{\bF} = \Z_p$, $\varpi_{\bF}=p$, $q_{\bF}=p$ and $\mathrm{val}(p^n\frac{a}{b}) = p^{-n}$, for any $p\nmid a,b$ and $n\in \Z$. 

Let $G = \bfG(\bF)$ be the group of $\bF$-rational points of a connected reductive $\bF$-group $\bfG$. 
Denote by $G_{\adj}$ its adjoint group, by $G_{\der}$ the derived subgroup of $G$, by $G_{\smc}$ the simply-connected cover of $G_{\der}$, and similarly for $\bfG$.
For example, if $\bfG = GL_n$, then $\bfG_{\adj} = PGL_n$ and $\bfG_{\smc} = \bfG_{\der} = SL_n$. For another example, if $\bfG = \bfU_n^{E,H}$, then $\bfG_{\adj} = PU_n^{E,H}$ and $\bfG_{\smc} = \bfG_{\der} = SU_n^{E,H}$.


\begin{thm} \label{thm:BT-1}
There exists a unique polysimplicial complex $\mB := \mB(G)$, called the reduced Bruhat-Tits building of $G$, which satisfies the axioms of \cite[\S4]{KP23}. In particular:
\begin{itemize}
\item The complex $\mB$ is a building of dimension $d:= \dim(\mB) = \mathrm{rank}_{\bF}(G_{\adj})$ (see \cite[Def.\ 1.5.5]{KP23}). It is a pure, locally-finite, contractible complex, and is infinite unless $d=0$.
\item $G$ acts by polysimplicial automorphisms on $\mB$, and in a strongly-transitive way, i.e.\ the action is transitive on pairs of a chamber and an apartment containing it in $\mB$. 
\item A surjective homomorphism with central kernel $G\to G'$, induces an equivariant isomorphism $\mB(G) \cong \mB(G')$. 
In particular, $\mB(G) \cong \mB(G_{\adj}) \cong \mB(G_{\der}) \cong \mB(G_{\smc})$, canonically, and the action of $G$ on $\mB(G)$ factors through $G_{\adj}$.
\end{itemize}
\end{thm}

Let $\mB = \mB(G)$ as above. For any $k=0,\ldots,d$, denote by $\mB(k)$ the set of $k$-dimensional faces ($k$-faces, for short) in $\mB$. For instance, $\mB(0)$ is the set of vertices of $\mB$, and $\mB(d)$ is the set of chambers of $\mB$. The building $\mB$ admits a type function $t\colon \mB(0) \to [d] = \{0,\ldots,d\}$, such that $t|_{\mC}$ is bijective for each $\mC \in \mB(d)$. Fix once and for all such a type function. For any face $\sigma \in \mB$, define its type to be $t(\sigma) = \{t(v) \,|\, v\in \sigma\} \subseteq [d]$.

For any subgroup $H\leq G$ and any $\sigma \in \mB$, denote by $H_\sigma$ the (set-wise) stabilizer of $\sigma$ in $H$. Say that $H$ acts in a type-preserving way if $t(h.\sigma) = t(\sigma)$, for any $h\in H$ and $\sigma \in \mB$. Note that if $H$ is type-preserving then $H_\sigma$ is also the point-wise stabilizer of $\sigma$. 

The main purpose of Bruhat--Tits theory is to study compact subgroups of $G$. To do so, we need to introduce the following subgroup $G^1$ of $G$.

\begin{defn} \label{defn:BT-G1} \cite[Lemma~2.6.16]{KP23}
Define 
\[
G^1 = \bigcap\nolimits_{\chi \in X^*(\bfG)} \ker(\mathrm{val}\circ \chi),
\]
where $X^*(\bfG) = \mathrm{Hom}_{\bF}(\bfG,\bG_m)$ is the group of $\bF$-algebraic characters of $\bfG$. 
\end{defn}

Note that $\bfG$ is semisimple if and only if $X^*(\bfG) = \{0\}$ if and only if $G^1 = G$.

\begin{prop} \label{prop:BT-G1} \cite[Axiom~4.1.2, Theorem~1.1.15]{KP23}
$G^1_\tau$ is an open compact subgroup of $G$ for any $\tau \in \mB$ and any compact open subgroup of $G$ is contained in $G^1_\tau$ for some $\tau \in \mB$. Consequently, $G^1_v$ is a maximal open compact subgroup of $G$ for any vertex $v\in \mB(0)$. 
\end{prop}

If $G^1$ acts in a type-preserving way, any maximal open compact subgroup of $G$ is of the form $G^1_v$, for some vertex $v\in \mB$. However, this is not true in general, for example in $G = PGL_2(\Q_p)$, the (set-wise) stabilizer of an edge is a maximal compact subgroup.

\begin{cor} \label{cor:BT-G1-max-comp}
If $\bfG$ is $\bF$-anisotropic modulo its center (e.g.\ if it is a minimal Levi subgroup or a torus), then $G^1$ is the unique maximal compact subgroup of $G = \bfG(\bF)$.
\end{cor}

\begin{proof}
By Theorem \ref{thm:BT-1} the Bruhat-Tits building of $G$ is a $0$-dimensional contractible complex, hence comprised of a single vertex, $\mB = \{v\}$, therefore $G = G_v$ and $G^1 = G^1_v$, and by the above Proposition we get that $G^1$ is the unique maximal compact subgroup.
\end{proof}


Next we wish to define the subgroup $G^0$ of $G$ \cite[Definition~2.6.23]{KP23}, which is a normal open subgroup of $G$ and a finite index subgroup of $G^1$ \cite[Lemmas~2.6.24~and~2.6.29]{KP23}, and use it to define the parahoric and Iwahori subgroups of $G$. 

Denote by $\bK/\bF$ the maximal unramified extension in a fixed seperable closure of $\bF$. By \cite[\S2.3]{KP23}, the base change $\bfG_{\bK} = \bfG \times_{\bF} \bK$ is a quasi-split group over $\bK$.

\begin{defn} \label{defn:BT-G0}
Let $G^\n$ be the image of $\bfG_{\smc}(\bF)$ in $\bfG(\bF)$, and define $G^\n \leq G^0 \leq G$, in steps:
\begin{enumerate}
\item For $\bfG$ a torus, define $G^0 = \bfG(\bK)^0 \cap G$, where $\bfG(\bK)^0 = N_{\bL/\bK}(\bfG(\bL)^1)$, $\bL/\bK$ is a finite Galois extension which splits $\bfG$, $N_{\bL/\bK} \colon \bfG(\bL) \to \bfG(\bK)$ is the norm map, and $\bfG(\bL)^1$ is the unique maximal bounded subgroup of $\bfG(\bL)$.
\item For $\bfG$ quasi-split, define $G^0 = G^\n \cdot \bfT(\bF)^0$, where  $\bfT \leq \bfG$ a maximal torus. Similarly, define $\bfG(\bK)^0 = \bfG(\bK)^\n \cdot \bfT(\bK)^0$, where $\bfT \leq \bfG_{\bK}$ a maximal torus.
\item For $\bfG$ which is $\bF$-anisotropic modulo the center, define $G^0 = \bfG(\bK)^0 \cap G$.
\item For general $\bfG$, define $G^0 = G^\n \cdot \bfM(\bF)^0$, where $\bfM$ is a minimal Levi subgroup of $\bfG$.
\end{enumerate}
A \textit{parahoric} subgroup of $G$ is a subgroup of the form $G_\tau^0$ for $\tau \in \mB$. An \textit{Iwahori} subgroup of $G$ is a minimal parahoric subgroup, or equivalently, $G^0_\mC$ for a chamber $\mC \in \mB$.
\end{defn}

\begin{rem} \label{rem:BT-sssc}
If $\bfG$ is a semisimple simply-connected group, then $G^\n = G^0 = G^1 = G$, in which case $G$ acts type-preservingly on $\mB$, any set-wise stabilizer is also a point-wise stabilizer, and the parahoric subgroups are precisely the stabilizers in $G$ of faces of the building. See other examples of $PGL$ and $PU$ in Examples \ref{exa:BT-B(PGL_n(F))}, \ref{exa:BT-B(PU_n(F))}, \ref{exa:BT-nqs-PU4}.
\end{rem}

\begin{thm} \label{thm:BT-G0} \cite[4.1.9 and 4.1.12]{KP23}
The subgroup $G^0$ acts in a type-preserving and strongly-transitive way on $\mB =\mB(G)$. Equivalently, if $(\mC,\mA)$ is a pair of a chamber and an apartment containing it, $\mI := G^0_\mC$ is the corresponding Iwahori subgroup, and $\mN := G^0_\mA$ is the (set-wise) stabilizer of $\mA$ in $G^0$, then $(\mI,\mN)$ is a saturated Tits system in $G^0$ (see \cite[Definition~1.4.1]{KP23}) whose associated (restricted) building is $\mB(G)$. 
\end{thm}

Bruhat-Tits theory highlights some vertices as being special, extraspecial (xs) and hyperspecial (hs), see \cite[Definitions 1.3.39 and 7.11.1]{KP23}. In general, every hs vertex is xs, and every xs vertex is special. If $\bfG$ is a split group then the three notions coincide. Hyperspecial vertices exists if and only if $G$ is unramified, i.e.\ $\bfG$ is quasi-split and $\bfG_\bK = \bfG \times_\bF \bK$ is split, in which case hs is equivalent to xs. The following result follows from \cite[Propositions 1.3.43 and 7.11.4, Remark 7.11.2]{KP23}.

\begin{prop} \label{prop:BT-special}
There exists an extraspecial vertex in every chamber of $\mB$. 
\end{prop}


\begin{defn} \label{defn:BT-special}
For $v \in \mB(0)$ a special (resp.\ xs/hs) vertex, we call $G^1_v$ a special (resp.\ xs/hs) maximal compact subgroup, and  $G^0_v$ a special (resp.\ xs/hs) maximal parahoric subgroup. We fix an xs (assumed hs as well if $G$ is unramified) vertex $v_0 \in \mB(0)$, and denote its maximal compact subgroup by $K = G^1_{v_0}$. 
\end{defn}


\begin{rem}
1. Since $G$ acts transitively on the chambers of $\mB$, to understand the conjugacy classes of special and xs/hs maximal parahoric subgroups, it suffices to understand these types of vertices in a given chamber $\mC \in \mB$ and their $G_{\mC}$-orbits. 

2. If $\bfG$ is an almost-simple, adjoint group, which is either split (e.g.\ $PGL_n$) or an odd dimensional projective unitary group, then by \cite[Lemma~1.3.75~and~Proposition~6.6.2]{KP23}, there is a single $G$-orbit of the extraspecial vertices (hyperspecial if the group is unramified).  

3. In our main case of interest, where $\bfG = \bfU_n^{E,H}$ and $\bF = F_{\p}$, Table~\ref{tab:types-actions} describes the structure of the chamber in each of the possibilities in lemma \ref{lem:U-points}. Much of the information is from \cite[Table~6.4.1, \S10.7b,c]{KP23}, and the intersection of the action of $G = \bfU_n^{E,H}(\bF)$, with the automorphisms of a chamber can be found by computing the image of $\pi_1(G)_I^\Frob$ in $\pi_1(G_{\ad})_I^\Frob$ (using Corollary 11.7.5 to interpret $\pi_1(G_{\ad})_I^\Frob$ as the setwise stabilizer of a chamber in $\bfG_{\ad}(\bF)$ modulo its pointwise stabilizer). We note that there is a unique extraspecial (hyperspecial if the group is unramified) orbit in all cases except when $v=\p$ is inert in $E$ and $n$ is even.   

\begin{table}[ht!]
\centering
\begin{tabular}{ccc|ccccc}
\multirow{2}{*}{$n$} & \multirow{2}{*}{Q-split?} & \multirow{2}{*}{$\p$} & \multicolumn{2}{c}{\textbf{Label}} & \multirow{2}{*}{\textbf{$G$-action}} & \multicolumn{2}{c}{\textbf{Orbits}} \\ &&&\cite{KP23} & \cite{Tit79} && xs & s \\
\midrule &  & split & $A_{n-1}$ & $A_{n-1}$ & transitive & 1 & 0 \\
\midrule \multirow{4}{*}{even} & \multirow{2}{*}{yes} & ur. & $C_{n/2}$ & ${}^2\!A'_{n-1}$ & trivial & 2 & 0 \\
& & ram. & $B_{n/2}^\vee$ & $B$-$C_{n/2}$ & involutive  & 1 & 0\\
\cmidrule{2-8} & \multirow{2}{*}{no} & ur. & $(C_{(n-2)/2}^\vee, C_{(n-2)/2})$ &  ${}^2\!A''_{n-1}$ & trivial & 2 & 0 \\
& & ram. & $BC_{(n-2)/2}$ & ${}^2\!B$-$C_{n/2}$ & trivial & 1 & 1\\
\midrule \multirow{2}{*}{odd} & \multirow{2}{*}{yes} & ur. & $(BC_{(n-1)/2}, C_{(n-1)/2})$ & ${}^2\!A'_{n-1}$ & trivial & 1 & 1 \\
& & ram. & $BC_{(n-1)/2}$ & $C$-$BC_{(n-1)/2}$ & trivial & 1 & 1 \\ 
\bottomrule
\end{tabular}
\vspace{\parskip}
\begin{tabular}{p{0.9\textwidth}}
\textbf{Note:} Labels in \cite{KP23} are in tables 1.3.3-4, here $C_1 = C^\vee_1 = B_1 = B^\vee_1 := A_1$, $B_2 := C_2$, $B^\vee_2 := C^\vee_2$. For labels in \cite{Tit79}, $B$-$C_2 := C$-$B_2$ and ${}^2\!B$-$C_{2} := {}^2\!C$-$B_2$. xs is extraspecial (hyperspecial if $G$ is unramified), and s is special but not xs.
\end{tabular}
\caption{Vertices in Chamber for $G = \bfU_n^{E,H}(F_{\p})$}
\label{tab:types-actions}
\end{table}
\end{rem}

Assume from now on that $\bfG$ is semisimple, hence $G^1 = G$. Let us fix some notations:

Let $\bfS \leq \bfG$ be a maximal $\bF$-split torus, $\bfM = \mathrm{Cent}_{\bfG}(\bfS)$ and $\bfN = \mathrm{Norm}_{\bfG}(\bfS)$ its centralizer and normalizer in $\bfG$, respectively, and let $S = \bfS(\bF)$, $M = \bfM(\bF)$ and $N = \bfN(\bF)$. By Bruhat-Tits theory (see \cite[Axiom~4.1.4]{KP23}), there exists a unique apartment $\mA$ in $\mB$ which corresponds to $\bfS$, such that $N = G_\mA$, the set-wise stabilizer of $\mA$, the action of $N$ on $\mA$ is by affine transformations, $M$ acts via affine translations, and $M^1$, the unique maximal compact subgroup of $M$, is the point-wise stabilizer of $\mA$. Fix $\mC \in \mA$ a chamber in the apartment, let $\mI = G^0_\mC$ be its corresponding Iwahori group, $v \in \mC$ an extraspecial vertex in the chamber, and $K = G_v$ its corresponding extraspecial maximal compact subgroup. Note that $\mN = G^0_\mA \leq N$ and $\mI \cap N = \mI \cap \mN = M^1 = K \cap M$. 

We now define several types of Weyl groups associated to $G$.

\begin{defn} \label{defn:BT-Weyl}
Define $\mW_f = N/M$, the finite Weyl group, $\mW = N/M^1$, the affine Weyl group, $\mW^a = \mN / \mN \cap \mI \leq \mW$, the type-preserving affine Weyl group,  $\mW_\mC$ the (set-wise) stabilizer of $\mC$ in $\mW$, $\mW_v = (K\cap N)/ M^1$ the stabilizer of $v$ in $\mW$, $\Lambda = M/M^1 \leq \mW$, the subgroup of translations, and $\Lambda^a = \Lambda \cap \mW^a$, the subgroup of type-preserving translations.\footnote{In the notations of \cite[\S6.6]{KP23}, $\mW = \widetilde{W}_1$ and $\mW^a = \widetilde{W}^{\mathrm{aff}}$, and in the notations of \cite[\S3.5]{Car79}, $\mW_f = W$, $\mW = W_1$, $\mW^a = W_{\mathrm{aff}}$ and $\mW_\mC = \Omega$.} Denote by $R$ the set of non-trivial elements $s \in \mW^a$, such that $\mI \sqcup \mI s \mI$ is a subgroup of $G^0$. 
\end{defn}

\begin{rem} \label{rem:BT-Weyl}
By \cite[\S3.5]{Car79}, the natural map $\mW_v \to \mW_f$ is an isomorphism, and the following direct sum decompositions holds: $\mW = \mW_\mC \ltimes \mW^a$, $\mW = \mW_f \ltimes \Lambda$ and $\mW^a = \mW_f \ltimes \Lambda^a$.
\end{rem}

By \cite[Remark~1.4.2]{KP23}, $(G^0,\mI,\mN,R)$ is the saturated Tits system of Theorem \ref{thm:BT-G0}. Hence, by \cite[\S1.4 and \S1.5]{KP23}, combined with  \cite[Proposition~6.36]{AB08}, we get:

\begin{prop} [Bruhat relations] \label{prop:BT-Bruhat-relations}
The group $\mW^a$ acts simply transitively on the chambers of $\mA$, $R \subset \mW^a$ is the set of reflections around the panels (codimension-one faces) of $\mC$, $\mW^a = \langle R \rangle$, and for any $w \in \mW^a$ and $s \in R$, the following relations holds:
\[
\mI s \mI \cdot \mI w \mI = \begin{cases}
\mI sw \mI & \ell(sw) = \ell(w) + 1 \\ \mI w \mI \sqcup \mI sw \mI  & \ell(sw) = \ell(w) - 1
\end{cases},
\]
where $\ell \colon \mW^a \to \N$ is the word length function with respect to the generating set $R$. 
\end{prop}

\begin{rem} \label{rem:BT-I+}
Denote $\mI^+ = G_\mC$, the set-wise stabilizer of the chamber $\mC$. Note that $\mI^+ = \mW_\mC \ltimes \mI$ and that $\mI^+ = N_G(\mI)$, the normalizer of $\mI$ in $G$. Hence for any $g = g_a\cdot g_c \in \mW$, $g_a \in \mW^a$ and $g_c \in \mW_\mC$, we have
\[
\mI g \mI = \mI g_a g_c \mI = \mI g_c g_a \mI = g_c \mI g_a \mI.
\]
\end{rem}

\begin{rem} \label{rem:BT-G0=G'I}
We observe that $G^0 = G_{\der} \cdot P$, for any parahoric subgroup $P \leq G^0$. The containment $G_{\der} \cdot P \subseteq G^0$ follows from the fact that $G_{\der} \leq G^\n$. In the other direction, by the Bruhat relations and the axioms of the Tits system for any $s\in R$, we get $s \mI s\mI \cap \mI s \mI \ne \emptyset$, and since $s \mI s\mI \subset G_{\der} \cdot \mI$, we get $\mI s \mI \subset G_{\der} \cdot \mI$, hence $G^0 = \mI \langle R\rangle \mI \subseteq G_{\der} \cdot \mI \subseteq G_{\der} \cdot P$.
\end{rem}

Let $X^*(\bfM) = \mathrm{Hom}(\bfM,\bG_m)$ be the group of $\bF$-algebraic characters, $X_*(\bfM) = \mathrm{Hom}_{\Z}(X^*(\bfM),\Z)$ the group of homomorphisms of $\Z$-modules and $\langle \cdot, \cdot \rangle \colon X_*(\bfM) \times X^*(\bfM) \to \Z$, the natural pairing. In the split case $\bfM = \bfS$, we identify $X_*(\bfS) = \mathrm{Hom}(\bG_m,\bfS)$, the group of $\bF$-algebraic cocharacters, such that $\chi(\varphi(t)) = t^{\langle \varphi , \chi \rangle}$. We note that $X^*(\bfM)$ (resp.\ $X_*(\bfS)$) is a finite index subgroup of $X^*(\bfS)$ (resp.\ $X_*(\bfM)$), and $X^*(\bfM) \cong X_*(\bfM) \cong X^*(\bfS) \cong X_*(\bfS) \cong \Z^d$, where $d = \dim \bfS$.  Following \cite[\S3.2]{Car79}, define the homomorphism $\mathrm{ord}_M \colon M \to X_*(\bfM)$ (compare with \cite[Notation~2.6.15]{KP23}), via the formula
\[
\langle \mathrm{ord}_M(g), \chi \rangle = (\mathrm{val}\circ \chi)(g), \qquad \forall \chi \in X^*(\bfM). 
\]
Note that $M^1 = \ker(\mathrm{ord}_M)$, the unique maximal compact subgroup, hence 
\[
\Lambda = M/M^1 \cong \mathrm{im}(\mathrm{ord}_M) \leq X_*(\bfM), \qquad \Lambda \cong \Z^d.
\]

Let $\Phi = \Phi(\bfS)$ be the root system of $\bfS$. Fix a positive root system $\Phi^+ \subset \Phi$. Let $\bfP$ be the minimal parabolic subgroup corresponding to $\Phi^+$, which contains $\bfM = \mathrm{Cent}_\bfG(\bfS)$ as a minimal Levi subgroups, let $\bfU^+$ be the unipotent radical of $\bfP$, and let $\bfU^-$ be the unipotent radical of the opposite parabolic subgroup corresponding to $\Phi^- = - \Phi^+$, the negative roots. Finally, define the following semi-group (denoted $\mZ$ in \cite[Equation~(5.2.2)]{KP23}),
\[
M^+ = \{m \in M \,|\, (\mathrm{val}\circ \alpha)(m) \leq 0, \quad \forall \alpha \in \Phi^+ \}
\]

\begin{defn} \label{defn:BT-positive-Weyl}
Define the positive Weyl chamber corresponding to $\Phi^+$ to be the semigroup
\[
\Lambda^+ = M^+/M^1 \cong \mathrm{ord}_M(M^+) \leq X_*(\bfM), \qquad \Lambda^+ \cong \N^d.
\]    
\end{defn}
 
We note that $\Lambda^+$ is a fundamental domain for the action of the finite Weyl group on $\Lambda$.

We are now able to state the Bruhat, Cartan, Iwahori and Iwasawa decompositions of the group $G$, following \cite[\S3.5]{Car79} (see also \cite[\S5]{KP23}). 

\begin{thm} [Bruhat decomposition] \label{thm:BT-Bruhat} 
$G = \mI \cdot \mW \cdot \mI = \bigsqcup_{w\in \mW} \mI w \mI$ (by which we mean that $w$ goes over any transversal of $\mW = N /M^1$).
\end{thm}

\begin{thm} [Cartan decomposition]  \label{thm:BT-Cartan} 
$G = K \cdot \Lambda^+ \cdot K = \bigsqcup_{\lambda \in \Lambda^+} K \lambda K$.
\end{thm}

\begin{thm} [Iwasawa decomposition] \label{thm:BT-Iwasawa}
$G = \bfP(\bF) \cdot \mW_f \cdot \mI = \bigsqcup_{w \in \mW_v} \bfP(\bF) w \mI$.  
\end{thm}

Denote $\mI^+_0 = (\bfU^+(\bF) \cap \mI)$ and $\mI^-_1 = (\bfU^-(\bF) \cap \mI)$, and note that for any $m \in M^+$,
\[
m \mI^+_0 m^{-1} \subseteq \mI^+_0 \qquad \mathrm{and} \qquad m^{-1} \mI^-_0 m \subseteq \mI^-_0.
\]

\begin{thm} [Iwahori decomposition] \label{thm:BT-Iwahori}
$\mI = \mI^-_1 \cdot M^0 \cdot \mI^+_0$.
\end{thm}





Using the Cartan decomposition we can introduce the notion of Cartan norms. Let $\Delta \subset \Phi^+$ be the set of simple roots. We note that $X^*(\bfM)$ embeds as a finite index subgroup in $X^*(\bfS)$, and denote $n = [X^*(\bfS):X^*(\bfM)]$, so that $n\cdot\chi \in X^*(\bfM)$ for any $\chi \in X^*(\bfS)$.

\begin{defn} \label{def:BT-Cartan-norm}
Let $w \colon \Delta \to \Q_{\geq 0}$ be a weight function. For any $g\in G$, by Theorem \ref{thm:BT-Cartan} let $\lambda_g \in \Lambda^+$, be such that $g \in K\lambda_gK$. Define the $w$-weighted Cartan norm to be 
\[
\|\cdot \|_w \colon G \to \Q_{\geq 0}, \qquad \|g\|_w = \sum_{\alpha \in \Delta} w(\alpha) \cdot \frac{1}{n}\cdot \langle \lambda_g , n \cdot \alpha \rangle.
\]
\end{defn}

The above definition of weighted Cartan norms extends both the Cartan and modified Cartan norms that we will use in \cite{DEP2} to analyze the growth of golden gate sets in our companion paper, as well as the R-norm in \cite{CP22}, which analyzes the cutoff of the Simple Random Walk (SRW) on Ramanujan complexes of type $\tilde{A}_d$.

Let us now give a few explicit examples, based on our cases of interest enumerated in Lemma \ref{lem:U-points} (see \cite{KP23}*{\S15} for the full details). We begin with the Bruhat-Tits building of the projective general linear groups (which is well known, see for example \cite{LSV05,AB08}).

\begin{example} \label{exa:BT-B(PGL_n(F))} 
Let $G = PGL_{d+1}(\bF)$. Then $\mB=\mB(G)$ is a $d$-dimensional simplicial complex, whose vertices are homothety classes $[L] = \{ \alpha L\,|\,\alpha\in\bF^{\times} \} $ of $\mO_\bF$-lattices $L$ in $\bF^{d+1}$. A collection of vertices $\sigma=\{ v_0,\ldots,v_k\}$ forms a face in $\mB$ if, possibly after reordering the $v_i$, there exist representatives $L_i\in v_i$ such that $L_0\subsetneq\ldots\subsetneq L_k\subsetneq\varpi_{\bF}^{-1}L_0$.  

For any vertex $v_0$, the number of edges containing it of the form $\{v_0,v_1\}$, such that $L_i\in v_i$,  $L_0\subsetneq L_1\subsetneq\varpi_{\bF}^{-1}L_0$ and $\dim_{\bF} (L_1/L_0) = k$, is ${n\choose k}_{q_{\bF}}$ (see (\ref{eq:D_n(q)})). In particular, the degree of any vertex is $D_n(q_{\bF})$. Similarly, the degree of any panel (i.e.\ the number of chambers containing it) is $q_{\bF}+1$. 

The group $G$ acts on the homothety classes of lattices by matrix multiplication and this action extends to the entire complex. The action is transitive on the vertices of $\mB$, and they are all hyperspecial. The subgroup $K = PGL_{d+1}(\mO_{\bF})$ is the stabilizer of the vertex $v_0=[\mO_{\bF}^{d+1}] \in \mB(0)$. 
Denoting $\val(g):=\min_{i,j}\val(g_{i,j})$ for $g\in GL_{d+1}(\mathbb{F})$,
we show now that 
\begin{equation} \label{eq:dist_PGL}
\mathrm{dist}\left(gv_{0},v_{0}\right)=-\val(g)-\val(g^{-1}).
\end{equation}
Both the l.h.s.\ and the r.h.s.\ are $K$-bi-invariant
(by $K=\mathrm{Stab}_{G}(v_{0})$ and $K=GL_{d+1}(\mathcal{O}_{\mathbb{F}})$,
respectively), and also invariant under scaling. Thus by the Cartan
decomposition $G=K\Lambda^{+}K$ it is enough to verify (\ref{eq:dist_PGL})
for $g=\lambda=\diag\left(\varpi^{m_{1}},\ldots,\varpi^{m_{d+1}}\right)$
with $0=m_{1}\leq\ldots\leq m_{d+1}$. The apartment $\mathcal{A}$ associated to the split maximal torus of diagonal matrices $S = \{\diag\left(s_1,\ldots,s_{d+1}\right)\,|\, s_i \in \bF^\times \}$,
contains $v_{0}$ and $\lambda v_{0}$, and the edges in $\mathcal{A}$
correspond to multiplying by $\diag(\varpi^{\delta_{1}},\ldots,\varpi^{\delta_{d+1}})$
with $\delta_{i}\in\{0,1\}$ (and not all $\delta_{i}$ the same),
so indeed $\dist(\lambda v_{0},v_{0})=m_{d+1}=-\val(\lambda)-\val(\lambda^{-1})$.

We remark that since $PGL_{d+1}$ is semisimple, we get that $G^1 = G$, and since $PGL_{d+1}$ is split, we get that $G^\n=G^0 = \{g \in G\,|\, \mathrm{val}(\det(g)) \in (d+1)\Z\}$. 

\end{example}

To describe the Bruhat-Tits building of unitary groups (in the tamely ramified case), we shall use the following Theorem, which summarizes certain claims from \cite[\S9, \S12]{KP23}.

\begin{thm} \label{thm:BT-tame}
Let $\bfG$ be a connected reductive $\bF$-group, $\bE/\bF$ a Galois field extension, $\mathrm{Gal}(\bE/\bF) = \langle \theta \rangle$ a cyclic group of order $\ell$, and $\bfH = R_{\bE/\bF}\bfG_{\bE}$, the restriction of scalars of the base change $\bfG_{\bE} = \bfG \times_{\bF} \bE$ from $\bE$ to $\bF$ (note that $\bfG = \bfH^\theta$). Assume that $\bE/\bF$ is tame, i.e.\ either unramified or $\ell$ is coprime to $q_{\bF}$.
Then:
\begin{enumerate}
\item $\theta$ induces a polysimplicial automorphism on $\mB(\bfH(\bF))$, also denoted $\theta$, which satisfies  $\theta(h.\sigma) = \theta(h).\theta(\sigma)$ for any $h \in \bfH(\bF)$ and $\sigma \in \mB(\bfH(\bF))$ \cite[Axiom~4.1.17]{KP23}.
\item There exists a $\bfG(\bF)$-equivariant polysimplicial bijection $\iota \colon \mB(\bfG(\bF)) \to \mB(\bfH(\bF))^\theta$, where $\mB(\bfH(\bF))^\theta = \{\tau \in \mB(\bfH(\bF)) \,|\, \theta(\tau) = \tau \}$ \cite[Theorems 9.2.7, 12.7.1]{KP23}. 
\item For any $\sigma \in \mB(\bfG(\bF))$, we have $\bfG(\bF)^0_\sigma = \bfH(\bF)^0_{\iota(\sigma)} \cap \bfG(\bF)$ \cite[\S9.2.6, \S12.8]{KP23}. 
\item If $\bE/\bF$ is unramified, and $u \in \mB(\bfH(\bF))^\theta$ is a special vertex, then $\iota^{-1}(u) \in \mB(\bfG(\bF))$ is a special vertex \cite[Proposition~9.9.2]{KP23}. 
\item If $\bE/\bF$ is unramified, $\bfG$ an unramified group, and $u \in \mB(\bfH(\bF))^\theta$ is a hyperspecial vertex, then $\iota^{-1}(u) \in \mB(\bfG(\bF))$ is a hyperspecial vertex \cite[Remark~7.11.8]{KP23}. 
\end{enumerate}
\end{thm}

\begin{example} \label{exa:BT-B(PU_n(F))}
Let $\bE/\bF$ be a tame quadratic extension, i.e.\ either unramified or such that $2\nmid q_\bF$, and let $\mathrm{Gal}(\bE/\bF)=\langle \theta \rangle$. Let $\wtd{\bfG}= R_{\bE/\bF}(PGL_n)_\bE$ be the restriction of scalars of the projective general linear group from $\bE$ to $\bF$. Let $\Phi \in GL_n(\bE)$ be a Hermitian matrix, where $(g^{*})_{i,j}=\theta(g_{j,i})$, and recall the projective unitary $\bF$-group w.r.t.\ $\bE/\bF$ and $\Phi$,
\[
\bfG = \bfPU_n^{\bE,\Phi} = \{g \in \wtd{\bfG} \,|\, g^* \Phi g = \Phi\}.
\]
Since $\bfG = \bfPU_n^{\bE,\Phi}$ is semisimple, we get that $G^1 = G$, and since $\bfG_\smc = \bfSU_n^{\bE,\Phi}$, we get that $G^\n = \{ g \in G \,|\, \det(g) \in \bfU_1^{\bE}(\bF)^n \}$. The subgroup $G^0$ can be calculated using Definition \ref{defn:BT-G0}, but we are only interested in the parahoric subgroups of $G$, namely $G^0_\sigma$ for $\sigma \in \mB$, for which we employ Theorem \ref{thm:BT-tame} as follows:
Denote $\bfH = R_{\bE/\bF}\bfG_\bE$, note that $\bfH(\bF) \cong PGL_n(\bE)$, and under this isomorphism we have $\theta(h) = \Phi^{-1}h^*\Phi$, hence $\bfG = \bfH^{\theta}$. By the above Theorem
\[
\mB(\bfPU_n^{\bE,H}(\bF)) \cong \mB(PGL_n(\bE))^{\theta}.
\]
By Claim (3) of the Theorem for any $\sigma \in \mB(G)$ we have $G^0_\sigma = H^0_{\iota(\sigma)} \cap G$ where note that for $H = \bfH(\bF) = PGL_n(\bE)$ then $H^0 = H^\n$ as descrbied in Example \ref{exa:BT-B(PGL_n(F))}.
Any vertex of $H = PGL_n(\bE)$ is hyperspecial, hence if $\bE/\bF$ is unramified (resp.\ and $\bfG$ is quasi-split) we get that any $\theta$-fixed vertex is a special (resp.\ hyperspecial) vertex of $\bfG(\bF)$.
\end{example}

\begin{example}\label{exa:BT-nqs-PU4}
Let $\bE/\bF$ be a tame quadratic extension, $n$ even, and $PU'_n(\bF) = \bfPU^{\bE,J}_n(\bF)$, where $J =  \diag(1, \dotsc, 1, a)$, for $(-1)^{n(n-1)/2}a \notin N_{\bE/\bF}(\bE^\times)$. Then $PU'_n(\bF)$ is non-quasisplit of rank $\tfrac{n}{2}-1$, i.e.\ its Bruhat-Tits buildings  is of dimension $\tfrac{n}{2}-1$. In particular, the Bruhat-Tits building $\mB = \mB(PU'_4(\bF))$ of  $PU'_4(\bF)$ is an infinite tree, such that:
\begin{enumerate}
\item If $\bE/\bF$ is unramified, then $\mB$ is a $(q_{\bF}^3 + 1)$-regular tree with all vertices extraspecial. 
\item If $\bE/\bF$ is ramified, then $\mB$ is a $(q_{\bF}^2+1,q_{\bF}+1)$-biregular tree, where only the vertices of degree $(q^2_{\bF}+1)$ are extraspecial. 
\end{enumerate}
\end{example}

For the purpose of proving class number one results via Prasad's volume formula (see Section \ref{sec:U}), we record the following remark which presents constructions of special maximal compact subgroups in $p$-adic unitary groups.

\begin{rem}\label{rem:BT-hermitian}
Explicit Hermitian forms $H_{\p}$ such that $U^{E_{\mf q},H_{\p}}(\mO_{\p})$ is special maximal compact can be found in \cite{gan2001exact}*{\S3}. For example, when $n$ is even and $G_{\p}$ is quasisplit, we may choose
\[
H_{\p} = \begin{pmatrix} & \m I_{n/2} \\\m I_{n/2} & \end{pmatrix} \text{ or } \begin{pmatrix} & a_{E_{\mf q}} \m I_{n/2} \\\bar a_{E_{\mf q}} \m I_{n/2} & \end{pmatrix}
\]
respectively for $\p$ unramified or $\p$ ramified with $a_{E_{\mf q}}$ a generator of the different ideal $\mf D_{E_{\mf q}/F_{\p}}$. When $G_{\p}$ is non-quasisplit and $E_{\mf q}/F_{\p}$ is unramified, we can similarly pick
\[
H_{\p} = \begin{pmatrix} & & & \m I_{(n-2)/2}  \\ & 1 & 0 & \\ & 0 & -\pi_{E_\mf q} & \\ \m I_{(n-2)/2} & & &\end{pmatrix}
\]
for uniformizer $\pi_{E_\mf q}$ of $F_\p$. 
When $E_\mf q/F_\p$ is ramified, then we can pick either 
\[
\begin{pmatrix} & & & a_{E_{\mf q}} \m I_{(n-2)/2}  \\ & 1 & 0 & \\ & 0 & -\beta & \\ \bar a_{E_{\mf q}} \m I_{(n-2)/2} & & &\end{pmatrix} \text{ or } \begin{pmatrix}\m I_{n-1} & \\& \beta\end{pmatrix}
\]
for the non-extraspecial or extraspecial case respectively, and where $\beta$ is an element of $\mO_\p$ that is not a norm from $E_\mf q$. The extraspecial case is not described in \cite{gan2001exact} and can be seen by computing points over the residue field. 

These examples work even for $\p|2$. Global Hermitian forms $H$ that localize to these specific $H_{\p}$ can be constructed by a Chinese remainder theorem argument as in \cite{MSG12}*{\S11.3.2}. 
\end{rem}

We end with the following instantiation of Corollary \ref{cor:U-transitive-criterion2}, namely, our second criterion for proving class number one (which avoids calculating the mass formula).

\begin{prop} \label{prop:BT-criterion2-class1}
Let $\bfG = \bfPU_{n}^{E,H}$ be as in Section \ref{sec:U}, with $h(E)=1$
and $\p=(p)$ a split prime not dividing $\disc H$. If 
\[
|\{g \in M_n(\mc O_E)\;|\; g^* H g = pH\}|/ |\bfG(\mO)|  = \textstyle{\sum\nolimits_{k=0}^{n} {n\choose k}_{q_{\p}}},
\]
then $\bfG$ is of class number one.
\end{prop}
\begin{proof}
Decomposing $p = \pi \overline{\pi}$, we have an embedding $\alpha\mapsto\widetilde{\alpha}\colon \mO_E\hookrightarrow F_\p$ with $v_\p(\widetilde{\pi})=1,v_\p(\widetilde{\overline{\pi}})=0$, which induces $G = \bfG(F_{\p}) \cong PGL_n(F_{\p})$, and $K = \bfG({\mc O}_{\p}) \cong PGL_n({\mc O}_{\p})$. If $g\in G$ and $v_0\in\mB(G)$ is the vertex with stabilizer $K$, we obtain from \eqref{eq:dist_PGL} and $H\in GL_n(\mO_E)$
\[
\dist\left(\widetilde{g}v_{0},\widetilde{g}\right)=-v_{\mathfrak{p}}(\widetilde{g})-v_{\mathfrak{p}}(\widetilde{g}^{-1})=-v_{\pi}(g)-v_{\pi}(H^{-1}g^{*}H)=-v_{\pi}(g)-v_{\overline{\pi}}(g).
\]
Denote $\Lambda_{\p}  = \bfG({\mc O}[1/\p])$, $C_{\p} = \bfG({\mc O}) = \Lambda_{\p} \cap K$, $S_{\p} = \{g \in \Lambda_{\p}\,|\, \mathrm{dist}(g.v_0,v_0) = 1\}$, and $\wh S_{\p} = \{g \in M_n(\mO_E) \,|\, g^* H g = pH\}$. For $g\in\wh S_{\p}$ we have $v_\pi(g),v_{\overline{\pi}}(g)\geq 0$, $v_\pi(g)+v_{\overline{\pi}}(g)\leq 1$, and examining the three cases we find that
\[\frac{g}{\pi}\in\begin{cases}
S_{\p} & v_{\pi}(g)=0,v_{\overline{\pi}}(g)=0\\
C_{\p} & v_{\pi}(g)=1,v_{\overline{\pi}}(g)=0\\
\tfrac{\overline{\pi}}{\pi}C_{\p} & v_{\pi}(g)=0,v_{\overline{\pi}}(g)=1.
\end{cases}\]
On the other hand we have $\pi C_\p,\overline{\pi} C_\p \subseteq \wh S_\p$, and every $g\in S_\p$ can be scaled (by a power of $\pi/\overline{\pi}$) so that $v_{\overline{\pi}}(g)=0$ and thus $v_\pi(g)=-1$, giving $\pi g\in\wh S_\p$. Corollary \ref{cor:U-transitive-criterion2} now proves the claim, as $|S_\p|/|C_\p| = |\wh S_\p|/|C_\p| - 2 = \sum\nolimits_{k=1}^{n-1} {n\choose k}_{q_{\p}}$.
\end{proof}

\section{Golden adelic groups}\label{sec:G}

In this section we introduce the notions of golden adelic groups. Let $\bfG = \bfU_n^{E,H}$, $\bfSU_n^{E,H}$ or $\bfPU_n^{E,H}$. For any prime $\p \in V_f$, denote $G_\p = \bfG(F_\p)$ and $K_\p = \bfG(\mO_\p)$. Let $\bfZ_{\bfG}$ be the center of $\bfG$. Denote by $K' = \prod_{\ell\in V_f}K'_{\ell}\leq \bfG(\widehat{\mO})$ a finite index adelic subgroup. 

\begin{defn}\label{def:G-gold} 
An adelic subgroup $K'\leq \bfG(\widehat{\mO})$ is said to be:
\begin{enumerate}
\item Almost golden if $\bfG(\A) = \bfG(F) \cdot G_{\infty}K'$.
\item Golden if it is almost golden and  $\bfG(F) \cap K' = \{I\}$.
\item Essentially golden if it is almost golden and $\bfG(F) \cap K' \subseteq \bfZ_{\bfG}(F)$.
\end{enumerate}
A golden (resp.\ almost, resp.\ essential) subgroup $K'$ is said to be \emph{super-golden} (resp.\ almost, resp.\ essential) at a prime $\p$, if $K'_\p$ is the stabilizer of a face in $\mB(G_\p)$ of dimension $>0$.

\end{defn}

\begin{rem}
The almost golden condition requires that $\bfG$ has class number one. Remark~\ref{rem:U-class-1} therefore makes almost golden, hence also essentially golden and golden, adelic subgroups quite rare. 
\end{rem}

\begin{defn}
Given $K' = \prod_{\ell} K'_\ell \leq \bfG(\wh{\mO})$ a finite index subgroup, $\p$ a prime, and write $(K')^\p = \prod_{\ell\ne\p} K'_\ell$. Define the level $K'$ congruence $\p$-arithmetic subgroup by
\begin{equation}
\Lambda_\p := \Lambda^{K'}_\p := \bfG(F) \cap (K')^{\p}. 
\end{equation}   
\end{defn}

Let $Z_{G_\p}$ be the center of $G_\p$ and $Z_{G_\p}^\spl$ the maximal split torus in $Z_{G_\p}$.

\begin{lem}\label{lem:gold-def-simp-tran}
Let $K'\leq \bfG(\widehat{\mO})$ be almost golden and let $\p$ be a prime such that $K'_\p = K_\p = \bfG(\mO_\p)=\Stab_{G_\p}(v_0)$ is an extraspecial maximal compact subgroup. Then:
\begin{enumerate}
\item $\Lambda_\p$ acts transitively on $G_\p/K_\p$ with stabilizer $\Gamma = K' \cap \bfG(F) \leq \bfG(\mO)$.
\item $\Lambda_\p/(Z_{G_\p}^\spl \cap \Lambda_\p)$ acts transitively on $G_{\p} .v_0$ with stabilizer $\Gamma/(Z_{G_\p}^\spl \cap \Gamma)$. 
\item $\Lambda_\p \cap Z^\spl_{G_\p}$ acts simply transitively on the fibers of $\Lambda_\p \to \Lambda_\p/(Z_{G_\p}^\spl \cap \Lambda_\p)$.
\end{enumerate}
In particular, when $K'$ is golden, the actions above are all simply transitive.
\end{lem}

\begin{proof}
Since $K_{\p}$ is a stabilizer of $v_0$, (1) follows from $G_{\p}=\Lambda_{\p}\cdot K'_{\p}$ and $\Lambda_{\p}\cap K'_{\p}=\Gamma$, which in turn follow from the assumption on the group $K'$ (similarly to lemma \ref{lem:U-class-p}). (2) follows since $v_0$ has stabilizer exactly $K_\p Z^\spl_{G_\p}$ and (3) is automatic. 
\end{proof}

\begin{rem}
We note that when $K'$ is essentially golden, then we get the additional special property that the quotients of the Bruhat-Tits building by normal congruence subgroups of $\Lambda_\p^{K'}$, for split primes $\p$, can also be described as Cayley complexes, i.e.\ clique complexes of Cayley graphs (see \cite{LSV05b,EP22}). It can also be used to give a neat combinatorial description for nonsplit $\p$ - see for an example the notion of a Cayley bigraph developed in \cite{EFMP23}.
\end{rem}

Let us now present several examples of golden adelic groups. Table 8.1 in \cite{Kir16} gives a very comprehensive list of unitary groups over $\Q$ with class number one, though in a form that takes some work to translate into what we need. 

\begin{example}[Previous Constructions]
In dimension $n=2$, several constructions were given in \cite{PS18,sarnak2015letter}. Let us present just one example of a golden adelic group:
\[
\bfG := \bfU_2^{\Q[\sqrt{-1}],I}/\Q,\qquad K' := \{ g\in \bfG(\widehat{\Z})\mid g\equiv I\mod 2\} .
\]
In dimension $n=3$, several constructions were given in \cite{EP22,EFMP23}. Here are two such examples of golden adelic groups:
\[
\bfG := \bfU_3^{\Q[\sqrt{-1}],I}/\Q, \qquad K' := \{ g\in \bfG(\widehat{\Z})\mid\forall i,\;g_{i,i}\equiv1\mod 2+2i\},
\]
\[
\bfG := \bfU_3^{\Q[\sqrt{-3}],I}/\Q, \qquad K' := \{ g\in \bfG(\widehat{\Z})\mid\forall i,\;g_{i,i}\equiv1\mod 3\}.
\]
Another construction from \cite{EFMP23} which is a variation on a construction of Mumford  \cite{mumford1979algebraic} uses the Hermitian positive definite matrix
\[
H := \left(\begin{smallmatrix} 3 & \lambda & \lambda\\ \bar{\lambda} & 3 & \lambda\\ \bar{\lambda} & \bar{\lambda} & 3 \end{smallmatrix}\right)\quad\mbox{where}\quad\lambda:=\tfrac{1+\sqrt{-7}}{2};
\]
the following is a golden adelic group:
\[
\bfG := \bfU_3^{\Q[\sqrt{-7}],H}/\Q,\qquad K' := \{ g\in \bfG(\widehat{\Z})\mid\forall i>j,\;g_{i,j}\equiv0\mod \lambda\}.
\]
\end{example}

Below we present our new constructions of golden adelic groups. 

\begin{rem} \label{rem:E8}
The motivation for the following construction comes from \cite{Wil12} and \cite{BF99}, which present elegant constructions of the $E_8$ lattice in $4$-dimensional spaces (field extensions in \cite{BF99}) over the Eisenstein quadratic imaginary field $E = \Q(\sqrt{-3})$. 
In fact, some of our other examples also appear to come from realizing the integers in other quadratic extensions as the $E_8$ lattice, suggesting a deeper, abstract reason why this is a fruitful place to look. 
\end{rem}

We start with the Eisenstein case of $E/F = \Q[\sqrt{-3}]/\Q$, which produces the example in Theorem \ref{thm:I-gold-Eis} from the Introduction.

\begin{prop} \label{prop:G-exm-Eis} 
For the $4\times 4$ Hermitian positive definite matrix,
\[
H_3 :=2\cdot\left(\begin{array}{cc}
I & A\\
A^* & I
\end{array}\right)\quad\mbox{where}\quad A:=\frac{\sqrt{-3}}{3}\cdot\left(\begin{array}{cc}
1 & 1\\
1 & -1
\end{array}\right),
\]
define 
\[
\bfG := \bfU_4^{\Q[\sqrt{-3}],H_3}/\Q.
\]
Then the following are essentially golden adelic groups:
\[
K_1^{(2)}:=\{ g\in \bfG(\widehat{\Z})\mid g\equiv I\mod 2\}, \qquad K_1^{(\sqrt 3)} :=\{ g\in \bfG(\widehat{\Z})\mid g\equiv I\mod{\sqrt 3}\}
\]
\end{prop}

\begin{proof}
We note that $H_3$ was chosen so that $\bfG$ is quasisplit at all places and $\bfG(\Z_p)$ is special at $p = 3$ and hyperspecial at all other $p$, as in Remark \ref{rem:BT-hermitian}. In particular $\Ram(H_3) = \emptyset$ so all the $\lambda_\ell$ in Definition \ref{def:U-mass} are $1$. Plugging in the $L$-values yields
\[
R(\bfG)=155520^{-1}.
\]
We may also compute that $|\bfG(\Z)| = 155520$. One method that runs in a few minutes on a personal laptop is to first minimally scale $H$ to have integral entries, note that the diagonal entries are then all $3$, find all\footnote{there are $240$ of these corresponding to the roots of an interpretation of $\mO_E^4$ as the $E_8$ lattice.} $v \in \mO_E^4$ of norm $3$, and finally build up all possible $\mO_E$-matrices preserving $H$ through choosing the rows one-by-one from this set of $v$. Therefore, by Proposition \ref{prop:U-class}, $G$ has class number one. 

Since $K_1^{(2)}$ is the kernel of the reduction map $r_2 : \bfG(\wh \Z) \to \bfG(\Z/2)$, it suffices to show that $r_2|_\bfG(\Z)$ is surjective with kernel contained in $Z_{\bfG}(\wh \Z)$. The prime $2$ is unramified in $G$, hence $\bfG(\Z/2\Z) \cong \bfU_4(\bF_2)$ and by a standard formula, $|\bfU_4(\bF_2)|= 77760 = |\bfG(\Z)|/2$. Another computer check finally shows that $\bfG(\Z)\cap\ker r_2 = \{ \pm 1 \}$.

For $K_1^{(\sqrt 3)}$, since $\bfG(\wh \Z)$ is special at $3$, $\bfG(\Z/3\Z) \cong \Sp_4(\F_3[x]/x^2)$ with reductive quotient $\Sp_4(\F_3)$. Furthermore, this composite map $G(\wh \Z) \to \Sp_4(\F_3)$ can be described exactly by reducing matrix entries mod $\sqrt{3}$. Therefore, by a standard formula, $[K : K_1^{(\sqrt 3)}] = |\Sp_4(\F_3)| = 51840 = |\bfG(\wh Z)|/3$. A computer check then gives that $G(\Z) \cap K_1^{(\sqrt 3)} = \{1, \z_3, \z_3^2\}$, which has $3$ elements.
\end{proof}

The group scheme $\bfU_4^{\Q[\sqrt{-3}],H_3}$ can also produce almost super-golden examples:

\begin{prop}\label{prop:G-exm-Eis-super}
Consider $\bfG = \bfU_4^{\Q\left[\sqrt{-3}\right],H_3}/\Q$ as in Proposition \ref{prop:G-exm-Eis}. Choose $X_2 \in \GL_4(\Z_2[\sqrt{-3}])$ such that
\[
X_2^*H_3X_2 \equiv \left(\begin{smallmatrix}
& & & 1 \\ & & 1 & \\ & 1 & & \\ 1 & & & 
\end{smallmatrix}\right)\pmod 4,\quad \text{e.g.}\quad X_2 := \tfrac{1}{2\sqrt{3}}
\left(\begin{smallmatrix}
-\sqrt{3}-3i & \sqrt{3}+i & \sqrt{3}-i & 4i\\
-\sqrt{3}-i & -\sqrt{3}-3i & 4i & -\sqrt{3}+i\\
0 & \sqrt{3}-3i & -\sqrt{3}-3i & 0\\
\sqrt{3}-3i & 0 & 0 & -\sqrt{3}-3i
\end{smallmatrix}\right).
\]
Define
\[
K^{(2),s} := \lf\{g \in \bfG(\wh \Z) \, \md| \, X_2^{-1}gX_2 \in \begin{pmatrix}
\mO_{E_2} &  \mO_{E_2} &  \mO_{E_2} &  \mO_{E_2} \\ 2\mO_{E_2} &  \mO_{E_2} &  \mO_{E_2} &  \mO_{E_2} \\ 2\mO_{E_2} &  \mO_{E_2} &  \mO_{E_2} &  \mO_{E_2} \\ 4\mO_{E_2} &  2\mO_{E_2} &  2\mO_{E_2} &  \mO_{E_2}
\end{pmatrix} \ri\}.
\]
Then $K^{(2),s}$ is an almost super-golden adelic group.
\end{prop}

\begin{proof}
Note that $X_2^{-1} \bfG(\Q_2) X_2$ then has a chamber corresponding to the lattice  $\mO_{E_2}^{4-i} \oplus 2 \mO_{E_2}^i$ for $0 \leq i \leq 3$. The non-special vertex corresponds to $i=1,3$ which are dual to each other under $X_2^* H_3 X_2$. Therefore lattices $\mO_{E_2}^4$ and $\mO_{E_2} \oplus 2\mO_E^2 \oplus 4\mO_2$ are hyperspecial connected through a two-step path through a non-special vertex and $K^{(2),s}$ exactly stabilizes both. 

We can also compute that $[K : K^{(2),s}] = [K_2 : K^{(2),s}] = 90 = \bfG(\Z)/1728$ (this is counting the choices of $\tau$ starting at the hyperspecial vertex $K$: there are $45$ choices for the first edge and $2$ per each of these for the second). Therefore, since $K$ is almost golden, $K^{(2),s}$ is almost golden if and only if $|G(\Z) \cap K^{(2),s}| = 1728$, which holds. 
\end{proof}

\begin{prop}\label{prop:G-exm-Eis-super-2}
Consider $\bfG = \bfU_{4}^{\Q\left[\sqrt{-3}\right],H_3}/\Q$ as in Proposition \ref{prop:G-exm-Eis}. Choose $X_3 \in \GL_4(\Z_3[\sqrt{-3}])$ such that for some $C \in (\Z/3)[\sqrt{-3}]$
\[
X_3^*H_3X_3 \equiv C\sqrt{-3}\left(\begin{smallmatrix}
 & & & 1 \\ & & 1 & \\ & -1 & & \\ -1 & & & 
\end{smallmatrix}\right)\!\!\! \pmod 3,\quad \text{e.g.}\quad X_3 := \left(\begin{smallmatrix}
4 & -3 & i \sqrt{3} & -2 i \sqrt{3} \\ 3 & 4 & -2 i \sqrt{3} & -i \sqrt{3} \\ -4 i \sqrt{3} & -i \sqrt{3} & -1 & -4 \\ -i \sqrt{3} & 4 i \sqrt{3} & 4 & -1   
\end{smallmatrix}\right).\]
Define
\[
K^{(3),s} := \lf\{g \in \bfG(\wh \Z) \, \md| \, X_3^{-1}gX_3 \in \begin{pmatrix}
\mO_{E_{\sqrt{-3}}} &  \mO_{E_{\sqrt{-3}}} &  \mO_{E_{\sqrt{-3}}} &  \mO_{E_{\sqrt{-3}}} \\ \mO_{E_{\sqrt{-3}}} &  \mO_{E_{\sqrt{-3}}} &  \mO_{E_{\sqrt{-3}}} &  \mO_{E_{\sqrt{-3}}} \\ 3 \mO_{E_{\sqrt{-3}}} &  3 \mO_{E_{\sqrt{-3}}} &  \mO_{E_{\sqrt{-3}}} &  \mO_{E_{\sqrt{-3}}} \\ 3 \mO_{E_{\sqrt{-3}}} &  3\mO_{E_{\sqrt{-3}}} &  \mO_{E_{\sqrt{-3}}} &  \mO_{E_{\sqrt{-3}}}
\end{pmatrix} \ri\}.
\]
Then $K^{(3),s}$ is an almost super-golden adelic group.
\end{prop}

\begin{proof}
Note that $X_3^{-1} \bfG(\Q_3) X_3$ then has a chamber corresponding to the lattices $\mO_{E_{\sqrt{-3}}}^{4-i} \oplus \sqrt{3} \mO_{E_{\sqrt{-3}}}^i$ for $0 \leq i \leq 3$. The non-special vertex corresponds to $i=2$. Therefore lattices $\mO_{E_{\sqrt{-3}}}^4$ and $\mO_{E_{\sqrt{-3}}}^2 \oplus 3\mO_{E_{\sqrt{-3}}}^2$ are special connected through a two-step path through a non-special vertex and $K^{(3),s}$ exactly stabilizes both. 

We then compute as in Proposition \ref{prop:G-exm-Eis-super} that $[K : K^{(3), s}] = [K_3 : K^{(3), s}_3] = 120 = |G(\Z)|/1296$ (there are $40$ choices for the first edge and $3$ per each of these for the second). It therefore suffices to check by computer that $|G(\Z) \cap K^{(3), s}| = 1296$ which holds. 
\end{proof}

We note that while the Hermitian positive definite matrix $H_3$ (which comes from the realization of the $E_8$ over $\Q[\sqrt{-3}]$, see \cite{Wil12}) gives golden and almost super-golden adelic groups, one can instead take the identity matrix $I$ and also get almost golden groups, though this time it is not clear whether almost super-golden groups can be obtained from this form.

\begin{prop} \label{prop:G-exm-Eis-stan} 
the following is an almost golden adelic group:
\[
\bfG := \bfU_{4}^{\Q[\sqrt{-3}],I}/\Q, \qquad K := \bfG(\wh \Z).
\]
\end{prop}

\begin{proof}
We note that $G$ is quasisplit at all places and $G(\Z_p)$ is hyperspecial at all $p \ne 3$.  
At $p=3$, we have that $G(\Z_3)$ is a maximal parahoic subgroup (see Lemma 5.8 in \cite{EFMP23}, it is stated for $n=3$ but the proof works verbatim for all $n \geq 2$).
By \cite{gan2001exact} we get that $\lambda_3 = \frac{3^{N(Sp_4/\bF_3))}\cdot Sp_4(\bF_3)}{3^{N(G/\bF_3)}\cdot \bfG(\bF_3)} = \frac{1}{5}$, and for any $p\ne 3$ we have $\lambda_p = 1$, hence $\lambda(G) = \frac{1}{5}$.
Since $G$ is an inner form of $\bfU_4^{\Q[\sqrt{-3}],H_3}$, and since $L(G)$ is invariant under taking inner forms, we get that $L(G) = 155520^{-1}$, hence $R(G) = 31104^{-1}$.
On the other hand, $G(\Z)$ is comprised of exactly the monomial matrices with non-zero coefficients in the unit group of the ring of integers, hence $G(\Z) \cong \langle \z_6 \rangle^4 \rtimes S_4$, and therefore $|G(\Z)| = 6^4\cdot 4! = 31104$.
Therefore, by Proposition \ref{prop:U-class}, $G$ has class number one. 
\end{proof}

Next, we present the rank-$8$ golden group of \cite{MSG12}. 

\begin{prop} \label{prop:G-exm-Eis-MSG} 
For the $8 \times 8$ Hermitian positive definite matrix
\[
{}^8\!H_3 := 
\begin{pmatrix}
H_3 & 0 \\
0 & H_3
\end{pmatrix},
\]
the following is an almost golden adelic group:
\[
\bfG := \bfU_8^{\Q[\sqrt{-3}], {}^8\!H_3}/\Q, \qquad K := \bfG(\wh \Z).
\]
\end{prop}

\begin{proof}
Then $\bfG$ is quasisplit at all places and $\bfG(\Z_p)$ is hyperspecial at all $p \neq 3$ and special at $p=3$. The result then follows from the main result of \cite{MSG12}. We can also compute that $R(\bfG) = 48372940800^{-1}$ and that $\bfU_8^{\Q[\sqrt{-3}], {}^8\!H_3}(\Z)$ is the $\bfU_4^{\Q[\sqrt{-3}], \!H_3}(\Z)$ from Proposition \ref{prop:G-exm-Eis} squared semidirect $\Z/2$ which has size $2 \cdot 155520^2 = 48372940800$. 
\end{proof}

Moving on to the Gaussian case $E/F = \Q[i]/\Q$:

\begin{prop}\label{prop:G-exm-Gauss}
For the $4 \times 4$ Hermitian positive definite matrix
\[
H_4 :=\left(\begin{array}{cc}
4I & A\\
A^* & 2I
\end{array}\right)\quad\mbox{where}\quad A := \left(\begin{array}{cc} 2+i & 1+i\\ -1+i & 2-i \end{array}\right),
\]
the following is an almost golden adelic group:
\[
\bfG := \bfU^{\Q[i], H_4}_4/\Q, \qquad K' := \bfG(\wh \Z).
\]
\end{prop}

\begin{proof}
$H_4$ is selected so that $G(\Z_p)$ is hyperspecial at all $p \neq 2$ and special at $p=2$ following Remark \ref{rem:BT-hermitian}. Note that the different ideal of $\Q_2[i]/\Q_2$ is $(2)$ so we can choose $H_4$ to be equivalent to the antidiagonal matrix of all ones over $\Z_2[i]$. 
Then the argument follows applying Proposition \ref{prop:U-class} as in Proposition \ref{prop:G-exm-Eis}: here, both $R(\bfG)^{-1}$ and $|\bfG(\Z)|$ are $46080$. 
\end{proof}

$H_4$ also produces some super-golden gate sets:

\begin{prop}\label{prop:G-exm-Gauss-super}
Consider $\bfG = \bfU_4^{\Q[i],H_4}/\Q$ as in Proposition \ref{prop:G-exm-Gauss}. Choose $X_2 \in \GL_4(\Z_2[i])$ such that
\[
X_2^*H_4X_2 \equiv \left(\begin{smallmatrix}
& & & 1 \\ & & 1 & \\ & 1 & & \\ 1 & & & 
\end{smallmatrix}\right)\!\!\pmod 4,\quad \text{for example}\quad X_2 := \left(\begin{smallmatrix}
-1 & -2 & -2 & -1 \\ 1 & 1 & -2 & 1 \\ 2 & 2 & 1-2 i & 2-i \\ -1-2 i & -i & 3 & -i 
\end{smallmatrix}\right).
\]
Define
\[
K' := \lf\{g \in \bfG(\wh \Z) \, \md| \, X_2^{-1}gX_2 \in \begin{pmatrix}
\mO_{E_{1+i}} &  \mO_{E_{1+i}} &  \mO_{E_{1+i}} &  \mO_{E_{1+i}} \\ \mO_{E_{1+i}} &  \mO_{E_{1+i}} &  \mO_{E_{1+i}} &  \mO_{E_{1+i}} \\ 2\mO_{E_{1+i}} &  2\mO_{E_{1+i}} &  \mO_{E_{1+i}} &  \mO_{E_{1+i}} \\ 2\mO_{E_{1+i}} &  2\mO_{E_{1+i}} &  \mO_{E_{1+i}} &  \mO_{E_{1+i}}
\end{pmatrix} \ri\}.
\]
Then $K'$ is an almost super-golden adelic group.
\end{prop}

\begin{proof}
As in Proposition \ref{prop:G-exm-Eis-super-2}, $X_2^{-1} \bfG(\Q_2) X_2$ exactly stabilizes a particular two-step path through a non-special vertex. In the wildly ramified setting, we additionally need to ensure that this two-step path is actually in $\mB(\bfG_{\Q_2})$ instead of intersecting barbs---this is guaranteed by the congruence condition on $X_2$ being mod $4$ instead of mod $2$. 

Then, arguing as in Proposition \ref{prop:G-exm-Eis},  $[K:K'] = [K_2:K'_2] = 30 = |\bfG(\Z)|/1536$ (there are $15$ choices for the first edge and $2$ per each of these for the second by standard formulas) and by computer check, $|\bfG(\Z) \cap K'| = 1536$. 
\end{proof}

Next we turn to what we call the Mumford case $E/F = \Q[\sqrt{-7}]/\Q$:


\begin{prop}\label{prop:G-exm-Mum}
For the $4 \times 4$ Hermitian positive definite matrix
\[
H_7 :=\left(\begin{array}{cc} 7I & A\\ A^* & 14I \end{array}\right) \quad\mbox{where} \quad A := \left(\begin{array}{cc} -3 \sqrt{-7} & -2 \sqrt{-7} \\ 2 \sqrt{-7} & -3 \sqrt{-7} \end{array}\right),
\]
the following is an almost golden adelic group:
\[
\bfG := \bfU_4^{\Q[\sqrt{-7}], H_7}/\Q, \qquad K' := \{g \in \bfG(\wh \Z) \,|\, g \text{ upper triangular}\pmod{\tfrac{1 + \sqrt{-7}}{2}} \}.
\]
\end{prop}

\begin{proof}
Here, $H_7$ was chosen so that $\bfG(\Z_p)$ is special at $p=7$ and hyperspecial at all other $p$ as in Remark \ref{rem:BT-hermitian}. Plugging in $L$-values gives $R(\bfG) = 5040^{-1}.$
As in Proposition \ref{prop:G-exm-Eis}, we can also compute that $|\bfG(\Z)| = 5040$, though the computation takes significantly longer to run\footnote{the norm-$7$ vectors again correspond to the $240$ roots of the $E_8$ lattice while the norm-$14$ vectors are the $2160$ of second-smallest length}. Therefore by Proposition \ref{prop:U-class}, $G$ has class number one. 
Next, in the embedding $G_2 \subseteq \GL_4(E_{(1 + \sqrt{-7})/2}) \times \GL_4(E_{(1 - \sqrt{-7})/2})$, projection onto the first coordinate is an isomorphism, so $K'_2$ is an Iwahori subgroup. This is the stabilizer of a chamber. 
Finally, we need to show that $[ \bfG(\Z) : \bfG(\Z) \cap K'] = [\bfG(\wh \Z) : K']$. The Iwahori in $\GL_2(\Z_2)$ has index $315$ by standard formulas for general linear groups over finite fields. Therefore, it suffices to check that $|\bfG(\Z) \cap K'| = 5040/315 = 16$, which holds. 
\end{proof}



Next we consider the case $E/F = \Q[\sqrt{-2}]/\Q$:

\begin{prop}\label{prop:G-exm-2}
For the $4 \times 4$ Hermitian positive definite matrix
\[
H_8 :=\left(\begin{array}{cc} 4I & A\\ A^* & 4I \end{array}\right) \quad\mbox{where} \quad A := \left(\begin{array}{cc} -\sqrt{-2} & -2+2\sqrt{-2} \\ -2-2 \sqrt{-2} & -\sqrt{-2} \end{array}\right),
\]
the following is an almost golden adelic group:
\[
\bfG := \bfU_4^{\Q[\sqrt{-2}], H_8}/\Q, \qquad K' := \lf\{g \in \bfG(\wh \Z) \,\md|\, g \equiv \left(\begin{smallmatrix} * & * & * & * \\ * & * & * & * \\ * & * & * & * \\ 0 & 0 & 0 & * \end{smallmatrix}\right) \pmod{1 + \sqrt{-2}} \ri \}.
\]
\end{prop}

\begin{proof}
As in the other examples, $H_8$ was chosen so that $\bfU^{E,H_8}(\Z_p)$ is special at $p=2$ and hyperspecial at all other $p$ as in Remark \ref{rem:BT-hermitian}. We also check\footnote{we again have that the norm-$4$ vectors with respect to $H_8$ correspond to roots of the $E_8$ lattice} that $R(G)^{-1} = 3840 = |\bfG(\Z)|$, hence $\bfG$ has class number one by Proposition \ref{prop:U-class}. 
Next, in the embedding $G_3 \subseteq \GL_4(E_{1 + \sqrt{-2}}) \times \GL_4(E_{1 - \sqrt{-2}})$, projection onto the first coordinate is an isomorphism. Therefore, $K'$ is a parahoric subgroup stabilizing an edge of type $(01)$. 
Finally, we check that $[\bfG(\Z) : \bfG(\Z) \cap K'] = [\bfG(\wh \Z) : K'] = [G_3 : K'_3]$. For this, $G_3/K'_3$ is projective $3$-space over $\F_3$ which has $40$ points and we can compute that $|\bfG(\Z) \cap K'| = 96 = 3840/40$. 
\end{proof}


We also have an example involving more complicated subgroups at primes over $2$:

\begin{prop}\label{prop:G-exm-2-2}
For the $4 \times 4$ Hermitian positive definite matrix
\[
H_2 := \left(\begin{array}{cc} 2I & \sqrt{-2}I \\ -\sqrt{-2}I & 2I \end{array}\right) = 
\left(\begin{smallmatrix} 2 & 0 & \sqrt{-2} & 0 \\ 0 & 2 & 0 & \sqrt{-2} \\ -\sqrt{-2} & 0 & 2 & 0 \\ 0 & -\sqrt{-2} & 0 & 2 \\     \end{smallmatrix}\right),
\]
the following is an almost golden adelic group:
\[
\bfG := \bfU_4^{\Q(\sqrt{-2}), H_2}/\Q, \qquad K := \bfG(\wh \Z).
\]
\end{prop}

\begin{proof}
Here we use Proposition \ref{prop:BT-criterion2-class1} on $G =\bfU_4^{\Q(\sqrt{-2}), H_2}$. Note that $E = \Q[\sqrt{-2}]$ is of class number one.  We consider $\p =3$, which is a split prime in $E$, $3 = (1 + \sqrt{-2})(1 - \sqrt{-2})$, and since $\det H = 4$ we get that $G({\mO}_{\p})$ is hyperspecial. We compute
\[
\sum_{i=0}^4 {4 \choose i}_3 = 1+40+130+40+1 = 212.
\]
A direct calculation via computer yields $|\bfG(\Z)| = 512$ and 
\[
|\{g\in M_4(\Z[\sqrt{-2}] \,|\, g^* H_2 g = 3 H_2\}| = 108544 = 512 \cdot 212.
\]
Hence, by Proposition \ref{prop:BT-criterion2-class1}, we get the claim.
\end{proof}

Finally, we present examples where $G$ is non-quasisplit at some places, so that the Bruhat-Tits building is a tree (see Example \ref{exa:BT-nqs-PU4}):

\begin{prop}\label{prop:G-nqs-Eis}
For the $4 \times 4$ Hermitian positive definite matrix
\[
H_{3R} := \diag(1,1,1,2),
\]
the following is an almost golden adelic group:
\[
\bfG := \bfU_4^{\Q[\sqrt{-3}], H_{3R}}/\Q, \qquad K:= \bfG(\wh \Z). 
\]
Note that $\bfG$ is non-quasisplit at $p=2,3$. The Bruhat-Tits building of $\bfG(\Q_2)$ is a $9$-regular tree, and of $\bfG(\Q_3)$ is a $(10,4)$-biregular tree.
\end{prop}

\begin{proof}
Non-quasisplitness can be tested by computing $\disc(H_{3R})$. In addition $\bfG(\wh \Z)$ was chosen as in Remark \ref{rem:BT-hermitian} so that it is extraspecial at $3$, special at $2$, and hyperspecial at all other places. Next, $\bfG(\Z)$ is easily computable as $\mu_6^4 \rtimes S^3$ where $\mu_6$ is the group of $6$th roots of unity and the $S_3$ acts on the first three coordinates. This has size $6^5 = 7776$ and we may also compute $R(\bfG) = 7776^{-1}$ so the result follows by Proposition \ref{prop:U-class}. 
\end{proof}


\begin{prop}\label{prop:G-nqs-Gauss}
For the $4 \times 4$ Hermitian positive definite matrix
\[
H_{4R} := \diag(1,1,1,3),
\]
the following is an almost golden adelic group:
\[
\bfG := \bfU_4^{\Q[i], H_{4R}}/\Q, \qquad K:= \bfG(\wh \Z). 
\]
Note that $\bfG$ is non-quasisplit at $p=2,3$, and the Bruhat-Tits buildings are a $(5,3)$-biregular and a $28$-regular trees, respectively.
\end{prop}

\begin{proof}
This is a similar argument to Proposition \ref{prop:G-nqs-Eis} with $R(G)^{-1} = |\bfG(\Z)| = 1536$. 
\end{proof}

\section{Hecke Operators}\label{sec:H}

The goal of this section is to provide the main upper bounds for the operator norm of Hecke operators acting on unitary irreducible representations in terms of the rate of decay of their matrix coefficients. These will be used in the next section to establish cutoff phenomena, and in the second paper \cite{DEP2} to study the covering rate of golden gate sets.

Let us fix some notations following Section \ref{sec:BT}. Let $\bfG$ be a connected semisimple algebraic group over a non-archimedean local field $\bF$, and let $G = \bfG(\bF)$. Let $\mB$ be the Bruhat-Tits building of $G$, $d = \dim \mB$ its dimension, $\mA \subset \mB$ a fixed fundamental apartment, $\mC \in \mA$ a fixed fundamental chamber, and $v_0 \in \mC$ a fixed extraspecial vertex. Let $\mI = G^0_\mC \leq \mI^+ =  G_\mC \leq K = G_{v_0} \leq G$ be the Iwahori subgroup corresponding to $\mC$, the stabilizer of $\mC$ and the stabilizer of $v$, respectively. Let $\Lambda^+$ be the positive Weyl chamber (Definition \ref{defn:BT-positive-Weyl}) and denote by $W = |\mW_f|$ the size of the finite Weyl group of $G$.

\begin{defn} \label{defn:H-flat-stable}
Let $P \leq G$ be a compact open subgroup and $S \subset G$ a finite set.
\begin{enumerate}
\item Say that $P$ is a flat stabilizer, if it is the stabilizer in $G$ of a finite collection of faces contained in a single apartment, i.e.\ $P = G_X = \mathrm{Stab}_G(X)$ for some $X \subset \mA$, $|X|<\infty$.
\item Say that $S$ is a $P$-stable set, if $PSP=SP$. 
\item Say that $S$ is a $P$-translation set, if $S \subset P  \Lambda^+  P$.
\end{enumerate}
\end{defn}

\begin{rem}
The above conditions are not too hard to satisfy:
\begin{enumerate}
\item If $\bfG$ is semisimple and simply connected, then by Remark \ref{rem:BT-sssc}, parahoric subgroups are the same as stabilizers of faces, hence they are flat stabilizers. 
\item By the Cartan decomposition (Theorem \ref{thm:BT-Cartan}), any set is a $K$-translation set.
\end{enumerate}
\end{rem}

For a unitary irreducible representation $\pi$ of $G$, we denote by $\sigma(\pi)$ its rate of decay of matrix coefficients: namely, it is the infimum among all $\sigma \geq 2$, such that all of its matrix coefficients coming from $K$-finite vectors are in $L^{\sigma}(G/Z_G)$, where $Z_G$ is the center of $G$. For short, we simply call $\sigma(\pi)$ the rate of $\pi$. Note that $\pi$ is tempered iff $\sigma(\pi) =2$ \cite{Haagerup1988}.

We now state the main result of this section:

\begin{thm} \label{thm:H-General-bound}
Let $P \leq G$ be a compact open subgroup and $S \subset G$ a finite $P$-stable set such that $|S|\geq 2$ and $| S^{-1}S \cap P | \ll 1$. Assume that either:
\begin{enumerate}
\item $P$ contains the Iwahori subgroup $\mI$ and $S\nsubseteq \mI^+$, or
\item $P$ is a flat stabilizer, $S$ is a $P$-translation set and $S \nsubseteq P$.
\end{enumerate}
If $\pi$ is a unitary irreducible representation of $G$, with rate $\sigma = \sigma(\pi)$, then
\[
\|\one_{P S P}\mid_{\pi^P} \|_{op} \ll (\log |S|)^{C+d} |S|^{\frac{\sigma - 1}{\sigma}}, 
\]
where $C = W$ in case 1, and in case 2 it is the constant $C_Q$ determined in Corollary \ref{cor:H-deep-Iwahori-bound}. 
\end{thm}

\begin{rem}
If fact, in case 1 of the above Theorem, if $P \subseteq K$, and denote by $W_P$ the size of the stabilizer of $v_0$ in the finite Weyl group, then we may take $C = W/W_P$. For example, if $P$ is a hypersepcial compact subgroup, then $C = 1$. 
\end{rem}

We start with a special case of the above Theorem, where $P = \mI$ and $S = \{g\}$.

\begin{prop} \label{prop:H-Iwahori-bound}
Let $g\in G \setminus \mI^+$. For any $\sigma \geq 2$, if $\pi$ is a unitary irreducible representation of $G$ whose matrix coefficients are in $L^{\sigma+\eps}$ mod center for all $\eps >0$, then
\[
\| \one_{\mI g \mI} \mid_{\pi^\mI} \|_{op}
\ll (\log |\mI g \mI / \mI|)^{dW} \cdot |\mI g \mI / \mI|^{\frac{\sigma -1}{\sigma}}.
\]
\end{prop}

\begin{rem}
We note that a slightly weaker bound can be deduced from the works of \cite{LP16,kamber2016lp}, namely, $\| \one_{\mI g \mI} \mid_{\pi^\mI} \|_{op} \lesssim |\mI g \mI / \mI|^{\frac{\sigma -1}{\sigma}}$. 
\end{rem}

\begin{proof}
By the Bruhat decomposition (Theorem \ref{thm:BT-Bruhat}) and Remark \ref{rem:BT-I+}, we may assume without loss of generality that $g \in \mW^a$. By the Bruhat relations (Proposition \ref{prop:BT-Bruhat-relations}) and Remark \ref{rem:BT-Weyl}, $g = s_1s_2\ldots s_r \lambda$,  where $s_i \in R$, $s_1s_2\cdots s_r \in \mW_f$ and $\lambda \in \Lambda^+$, and the following identity for convolution of operators holds:
\[ 
\one_{\mI sw \mI}=\begin{cases} 
\one_{\mI s \mI}\one_{\mI w \mI} & \ell\left(sw\right)=\ell\left(w\right)+1\\ \one_{\mI s \mI}^{-1}\one_{\mI w \mI} & \ell\left(sw\right)=\ell\left(w\right)-1
\end{cases}, \qquad \forall s \in R,\; w\in \mW^a.
\]
Denote $q_s = |\mI s \mI / \mI|$, for any $s \in R$. Then $\one_{\mI s_i \mI}$ is a root of $x^2-(q_{s_i}-1)x+q_{s_i}$, and we have $1/q_{s_i}\le\| \one_{\mI s_i \mI}^{\pm1}\|_{op}\le q_{s_i}$, thus $\| \one_{\mI g \mI} \mid_{\pi^\mI} \|_{op}$ is bounded by $\| \one_{\mI \lambda \mI} \mid_{\pi^\mI} \|_{op}$ times a constant (namely, $\max_i{q_{s_i}^{\pm h}}$ where $h = h(\mW_f)$ is the Coxeter number). In particular, taking $\pi$ to be the trivial representation this  shows that $|\mI \lambda \mI/\mI|\ll|\mI g \mI/\mI|$, so it is enough to prove the Proposition for $g = \lambda \in \Lambda^+$.

Let $\lambda_1,\ldots,\lambda_d \in \Lambda^+$ be a basis for the semigroup, and write $\lambda = \sum_i m_i \lambda_i$, $m_i \geq 0$. It follows from $g \notin \mI^+$ that at least one $m_i$ is positive. If $\ell$ is the word length function then $\ell(\lambda) = \sum_i m_i \ell(\lambda_i)$, and by the Bruhat relations we get that $\one_{\mI \lambda \mI} = (\one_{\mI \lambda_1 \mI})^{m_1} \circ \ldots \circ (\one_{\mI \lambda_d \mI})^{m_d}$ and $\mI \lambda \mI = \mI \lambda_1^{m_1} \mI \cdots \mI \lambda_d^{m_d} \mI = (\mI \lambda_1 \mI)^{m_1} \cdots (\mI \lambda_d \mI)^{m_d}$. In particular, $|\mI \lambda \mI/\mI| = k_1^{m_1}\cdot\ldots\cdot k_d^{m_d}$ for $k_i = |\mI \lambda_i \mI / \mI|$, and $m_1\cdot\ldots\cdot m_d \ll(\log |\mI \lambda \mI/\mI|)^d$. Hence it is enough to prove the Proposition for $\lambda = \lambda_1^m$. Write $T= \one_{\mI \lambda_1 \mI}$ and note that $\one_{\mI \lambda \mI} = T^m$ and $|\mI \lambda \mI / \mI| = k^m$, where $k = |\mI \lambda_1 \mI / \mI|$ and $m \ge 1$. 

By \cite[Prop.\ 4.5]{LLP20}, the Iwahori-operator $T$ is $W$-normal. Denoting $\lambda = \sup\{|z| \,|\, z \in \mathrm{Spec}(T\mid_{\pi^\mI})\}$, by \cite[Prop.\ 4.1]{Par20} we get $\|T^m|_{\pi^\mI}\|_{op} \ll m^W \lambda^m$ for $m\ge 1$ (and $\|T^0|_{\pi^\mI}\|_{op} =1$), so it is only left to prove that $\lambda \leq k^{\frac{\sigma -1}{\sigma}}$. Let $v\in\pi^{\mathcal{I}}$
be a $\lambda$-eigenvector of $T$, and consider its matrix coefficient
$\varphi\colon g\mapsto\left\langle gv,v\right\rangle $. This is
a bi-$\mathcal{I}$-invariant function, and denoting $\mathcal{I}=G^{0}_{\mathcal{C}}$,
we consider it as a function on the chambers of $\mathcal{B}$
via $\varphi(g\mathcal{C})=\varphi(g)$. Since $\sigma=\sigma(\pi)$,
it is in $\bigcap_{\varepsilon>0}L^{\sigma+\varepsilon}(\mathcal{B}(d))$, and it is also a $\lambda$-eigenfunction of $T$ acting on $L^{2}(G/\mathcal{I})$.
Normalizing $v$ so that $\varphi(\mathcal{C})=1$, we claim that
\begin{equation}
\sum_{\mathcal{C}'\in T^{j}\left(\mathcal{C}\right)}\left|\varphi\left(\mathcal{C}'\right)\right|^{\sigma+\varepsilon}\geq\biggl|\frac{\lambda^{\sigma+\varepsilon}}{k^{\sigma-1+\varepsilon}}\biggr|^{j}.\label{eq:holder}
\end{equation}
We apply induction, where $j=0$ is clear. Since $\lambda_1 \in \Lambda^+$ is a translation, by \cite[Thm.\ 5.6]{LLP20} $T$ is \textit{collision-free}, i.e.\ in the walk it
describes on $\mathcal{B}$ there is at most one possible path from each face
to another. In particular, $T^{j+1}(\mC)$ are $k^{j+1}$ distinct faces, so that
\begin{align*}
\sum_{\mathcal{C}''\in T^{j+1}\left(\mathcal{C}\right)}\!\!\left|\varphi\left(\mathcal{C}''\right)\right|^{\sigma+\varepsilon} & =\sum_{\mathcal{C}'\in T^{j}\left(\mathcal{C}\right)}\sum_{\mathcal{C}''\in T\left(\mathcal{C}'\right)}\!\!\left|\varphi\left(\mathcal{C}''\right)\right|^{\sigma+\varepsilon}\overset{(*)}{\geq}\frac{1}{k^{\sigma-1+\varepsilon}}\sum_{\mathcal{C}'\in T^{j}\left(\mathcal{C}\right)}\Big(\sum_{\mathcal{C}''\in T\left(\mathcal{C}'\right)}\!\!\left|\varphi\left(\mathcal{C}''\right)\right|\Big)^{\sigma+\varepsilon}\\
 & \geq\frac{|\lambda|^{\sigma+\varepsilon}}{k^{\sigma-1+\varepsilon}}\sum_{\mathcal{C}'\in T^{j}\left(\mathcal{C}\right)}\left|\varphi\left(\mathcal{C}'\right)\right|^{\sigma+\varepsilon}\geq\left|\frac{\lambda^{\sigma+\varepsilon}}{k^{\sigma-1+\varepsilon}}\right|^{j+1}
\end{align*}
where $(*)$ is by H\"{o}lder's inequality for the H\"{o}lder conjugates $\sigma+\varepsilon,\frac{\sigma+\varepsilon}{\sigma-1+\varepsilon}$.
Using $\varphi\in\bigcap_{\varepsilon>0}L^{\sigma+\varepsilon}(\mathcal{B}(d))$ together with (\ref{eq:holder}) and the collision-freeness of $T$ we obtain
\[
\infty>\left\Vert \varphi\right\Vert ^{\sigma+\varepsilon}_{\sigma+\varepsilon}\geq\sum^{\infty}_{j=0}\sum_{\mathcal{C}'\in T^{j}\left(\mathcal{C}\right)}\left|\varphi\left(\mathcal{C}'\right)\right|^{\sigma+\varepsilon}\geq\sum^{\infty}_{j=0}\left|\frac{\lambda^{\sigma+\varepsilon}}{k^{\sigma-1+\varepsilon}}\right|^{j}
\]
for every $\varepsilon>0$, which forces $\left|\lambda\right|\leq k^{\frac{\sigma-1}{\sigma}}$ as desired.
In total, we have
\[
\|\one_{\mI \lambda \mI}\|_{op} \leq \prod\nolimits_{i=1}^{d} \|\one_{\mI \lambda_i \mI}^{m_i}\|_{op} \ll \prod\nolimits_{i=1}^{d} m_i^W k_i^{m_i \frac{\sigma-1}{\sigma}} \leq (\log|\mI \lambda \mI/ \mI|)^{dW}|\mI \lambda \mI/\mI|^{\frac{\sigma-1}{\sigma}}.\qedhere
\]
\end{proof}

The next results will allow us to treat more general compact subgroups.

\begin{prop} \label{prop:H-coset-size}
Let $P \leq G$ be a flat stabilizer and $S \subset G$ a finite $P$-stable set. Then 
\[
|P \backslash PSP / P | \ll \left( \log |S| \right)^d.
\]
Furthermore, assuming $|S| \ll |PSP/P|$ (for example, if $|S^{-1}S\cap P|\ll 1$), we get
\[
\left( \log |S| \right)^{-d} \cdot |S| \ll \max_{s\in S}|PsP/P| \ll |S|.
\]
\end{prop}

\begin{proof}
Assume first that $X = \{\sigma_0\}$ is a single face, hence $P$ is a stabilizer of a face, and assume without loss of generality that $\sigma_0 \subset C$, i.e.\ $\mI \subset P$. Let $\ell(s) = \mathrm{dist}(s.\sigma_0,\sigma_0)$, $\ell(S) = \max_{s\in S}\ell(s)$ and let $s_m \in S$ such that $\ell(S) = \ell(s_m)$. Using the Bruhat decomposition, for any $g\in G$, denote by $w_g$ the unique Iwahori-Weyl element such that $g \in \mI w_g \mI$. We note that $\ell(g)$ is quasi-isometric to the Coxeter length of $w_g$ which also equals $\log_q|\mI w_g \mI/\mI|$. Combining this with Lemma \ref{lem:H-degrees}, we get 
\[
|PgP/P|\asymp |\mI g \mI/\mI| \asymp q^{\ell(g)}.
\]
Denote by $B_{\mB}(r) = \{\sigma \in \mB \,|\, \mathrm{dist}(\sigma,\sigma_0) \leq r\}$ and $B_{\mA}(r)= \{\sigma \in \mA \,|\, \mathrm{dist}(\sigma,\sigma_0) \leq r\}$ the balls of radius $r$ around $\sigma_0$ in the building and apartment, respectively (if $\sigma$ is not of the same size as $\sigma_0$ then we decree $\mathrm{dist}(\sigma,\sigma_0) = \infty$). Note that $P$ acts on $B_{\mA}(r)$ and that  $P.B_{\mA}(r) = B_{\mB}(r)$. Also note that since $\mA$ is Euclidean, the size of its balls is approximately their radius to the power of the dimension. Hence 
\[
|P \backslash B_{\mB}(r)| \ll |B_{\mA}(r)| \asymp r^d.
\]
On the one hand, since $PSP = SP$, we get
\[
q^{\ell(S)} =q^{\ell(s_m)} \asymp |Ps_mP/P| \leq |PSP/P| = |SP/P| \leq |S|. 
\]
On the other hand, since $S.\sigma_0 \subset B_{\mB}(\ell(S))$, we get from the above estimates that
\[
|P\backslash PSP/P| = |P \backslash PS.\sigma_0| \leq |P \backslash B_{\mB}(\ell(S))| \asymp \ell(S)^d \ll (\log_q |S|)^d.
\]
Next we consider a general finite collection of faces $X$. Define a distance function on finite collection of $\mB$, by $\mathrm{Dist}(Y,Z) = 0$, if $Z$ and $Y$ are not in the same $G$-orbit, and otherwise $\mathrm{Dist}(Y,Z) = \min\{\|g\|_0 \,|\, g \in G,\, g.Z = Y\}$, where $\|\cdot\|_0$ is the Cartan norm (it could also be the modified Cartan, for the asymptotic argument here it will not matter). Denote by $B^D_{\mB}(Y,r)$ and $B^D_{\mA}(Y,r)$ the balls around $Y$ of radius $r$, w.r.t.\ $\mathrm{Dist}$, in $\mB$ and $\mA$, respectively. Let $K$ be a special maximal compact subgroup of $G$ and without loss of generality assume that $P \subset K$ and $I \subset K$: hence $K.\mA = \mB$.

Let $\Omega$ be a finite set of representatives $K/P$ and note that $\mB = P. \left(\bigcup_{w \in \Omega} w. \mA \right)$. Denote $c = \max_{w \in \Omega} \mathrm{Dist}(X,w.X)$ and $\ell = \max_{s \in S} \mathrm{Dist}(X,s.X)$. Then $S.X \subset B^D_{\mB}(X,r)$, and by the triangle inequality, $B^D_{\mB}(X,\ell) \subset P. \left(\bigcup_{k_i \in \Omega} B^D_{k_i.\mA}(k_i.X,\ell+c) \right)$. From $P$-stability we get $PSP / P = SP / P \cong S.X$, and by arguing as before we get 
\[
|P \backslash PSP / P | = |P \backslash S.X| \leq |P \backslash B^D_{\mB}(X,\ell) | \leq \sum_{w \in \Omega} |B^D_{w.\mA}(w.X,\ell + c)| \ll \ell^d \asymp (\log |S|)^d,
\]
which completes the proof of the first identity.

For the second identity, we now have  $|S|\ll |PSP/P| = |SP/P|\le|S|$. Hence $|PsP/P| \leq |PSP/P| \leq |S|$, for any $s \in S$, which gives the right inequality. The left inequality follows from 
\[
|S| \ll |SP/P| = |PSP/P| = \sum_{PsP \in P \backslash PSP/P} |PsP/P| \leq |P \backslash PSP/P| \cdot \max_{s\in S} |PsP/P|
\]
together with the first identity $ |P \backslash PSP/P| \ll (\log |S|)^d$. This completes the proof.
\end{proof}

\begin{lem} \label{lem:H-degrees}
Let $Q\leq P \leq G$ be two fixed open compact subgroups, and let $\Omega \subset P$ be a finite transversal set such that $P = \Omega Q = Q \Omega$. Then for any $g \in G$,
\[
|\Omega| \cdot |QgQ/Q| \leq |P g P / P| \leq |\Omega|^3 \cdot |Q g Q / Q|,
\]
in particular $|QgQ/Q|\asymp |PgP/P|$, and for any $P$-spherical representation $V$ of $G$, 
\[
\| \one_{PgP} \mid_{V^P} \|_{op} \leq |\Omega|^2 \max_{w,w'\in \Omega} \|\one_{Qwgw'Q}\mid_{V^P} \|_{op}.
\]
\end{lem}

\begin{proof}
Let $\mu$ be a Haar measure on $G$, which is bi-invariant since $G$ is a reductive $p$-adic group. Note that $|X/H| = \frac{\mu(X)}{\mu(H)}$, for any compact open set $X \subset G$ and compact open subgroup $H \leq G$.  Hence for the first claim it suffices to prove $\mu(QgQ) \leq \mu(PgP) \leq |\Omega|^2 \cdot \mu(Q g Q)$, which follows from the fact that $QgQ \subset PgP = \Omega Q g Q \Omega$.

For the second claim we pick a transversal set $X \subset \Omega g \Omega$ for the space of double cosets $Q \backslash P g P / Q$. In particular, $PgP = \bigsqcup_{x\in X} QxQ$, hence $\one_{PgP} = \sum_{x\in X} \one_{QxQ}$. Therefore 
\begin{equation}
\begin{aligned}
\| \one_{PgP} \mid_{V^P} \|_{op} & \leq \sum_{x\in X} \|\one_{QxQ}\mid_{V^P} \|_{op}
\leq |X| \max_{x\in X} \|\one_{QxQ}\mid_{V^P} \|_{op} \\ &
\leq |\Omega|^2 \max_{w,w'\in \Omega} \|\one_{Qwgw'Q}\mid_{V^P} \|_{op}.   
\end{aligned}
\qedhere 
\end{equation}

\end{proof}

\begin{defn} \label{defn:H-geodesic}
For $r\in \N$, denote by $\Delta_r$ the convex hull in $\mA$ of $0$ and $r\cdot \lambda_\alpha/l_\alpha$, call it the fundamental $r$-level truncated sector, and denote $\mI_r = G^0_{\Delta_r} = \mathrm{stab}_G(\Delta_r) \cap \mI$, call it the $r$-deep Iwahori subgroup. Note that $\Delta_1 = C$ and $\mI_1 = \mI$.
\end{defn}

\begin{prop} \label{prop:H-deep-Iwahori-1}
Let $Q = \mI_r$ be the $r$-deep Iwahori subgroup for some $r$. For any $m \in M^+$, the (well defined) map $QmQ/Q \to \mI m \mI/ \mI$, $qmQ \mapsto qm \mI$, is a bijection.
\end{prop}

\begin{proof}
Let $\mI = \mI_1^{-} M^0 \mI_0^{+}$ be the Iwahori decomposition of Theorem \ref{thm:BT-Iwahori}. Observe that $Q_r = \mI^{-}_r M^0 \mI_0^{+}$ is the $r$-deep Iwahori subgroup (note that $M^0 \mI_0^{+}$ is the stabilizer of the full sector $\Delta = \bigcup_r \Delta_r$ in $\mI$). Note that for any $m \in M^+$ and $k \in \N$, then $m^{-1} \mI_k^0 m = \mI_k^0$,  $m^{-1} \mI_k^{+} m \supseteq \mI_k^{+}$ and  $m^{-1} \mI_k^{-} m \subseteq \mI_k^{-}$. Therefore $m^{-1} Q m \cap \mI \subset  Q$ and $QmI = \mI m \mI$. Injectivity follows from $m^{-1} Q m \cap \mI \subset  Q$, since for any  $q_1,q_2\in Q$ such that $q_1 m \mI = q_2 m \mI$, then $q = m^{-1} q_2^{-1} q_1 m \in m^{-1} Q m \cap \mI \subset Q$, hence $q_1 m Q = q_2 m Q$, and surjectivity clearly follows from $Q m \mI = \mI m \mI$.
\end{proof}

\begin{prop} \label{prop:H-deep-Iwahori-2}
Let $Q = \mI_r$ be the $r$-deep Iwahori subgroup for some $r$. Let $\lambda_1,\ldots,\lambda_d \in \Lambda^+$ be a basis for the semigroup, and for any $m\in M^+$, write $mM^1 = \lambda = \sum_i m_i \lambda_i \in \Lambda^+$, $m_i \geq 0$. Then $\one_{Q m Q}$  is collision-free and satisfies
\[
\one_{Q m Q} = (\one_{Q \lambda_1 Q})^{m_1}\circ \ldots \cdot\circ (\one_{Q \lambda_n Q})^{m_n}.
\]
\end{prop}

\begin{proof}
By the Bruhat decomposition of Theorem \ref{thm:BT-Bruhat}, each $\lambda \in \Lambda^+$ lie in a different double $\mI$-coset, hence $\Phi(q\lambda Q) = q \lambda \mI$, defines a bijection $Q\Lambda^+ Q/Q\rightarrow \mI\Lambda^+ \mI/\mI$. Therefore, by Proposition \ref{prop:H-deep-Iwahori-1}, $\Phi$ intertwines the branching operators $\one_{Q\lambda Q}$ and $\one_{\mI \lambda \mI}$ acting on compactly supported functions on $Q\Lambda^+Q/Q$ and $\mI \Lambda^+ \mI/\mI$ respectively. Thus, we have
\[
\one_{Q \lambda Q}=\Phi^{-1} \one_{\mI \lambda \mI} \Phi = \Phi^{-1} \prod\nolimits_{i=1}^{n} \one_{\mI \lambda_i \mI}^{m_i}\Phi = \prod\nolimits_{i=1}^{n}\one_{Q \lambda_i Q}^{m_i},
\]
(on $Q \Lambda^+ Q/Q$ and thus everywhere by $G$-equivariance of the Hecke action), and it is furthermore collision free since $\one_{\mI \lambda \mI}$ is. 
\end{proof}

Let us recall Bernstein's uniform admissibility Theorem \cite{bernstein1974all}.

\begin{thm}[\cite{bernstein1974all}] \label{thm:H-Bernstein}
For any open compact subgroup $Q\leq G$ there exists $N(G,Q) \in \N$, such that for any admissible irreducible representation $\pi$ of $G$, $\dim \pi^Q \leq N(G,Q)$.  
\end{thm}

The following Corollary generalizes Proposition \ref{prop:H-Iwahori-bound} to deeper Iwahori subgroups, for translation elements.

\begin{cor}\label{cor:H-deep-Iwahori-bound}
Let $Q = \mI_r$ be an $r$-deep Iwahori subgroup, $\sigma \geq 2$, and $C = N(G,Q)$ as in Theorem \ref{thm:H-Bernstein}. Then for any translation $g = m \in M^+$, $\lambda = mM^1 \in \Lambda^+$, and any unitary irreducible representation $\pi$ of $G$, with rate $\sigma = \sigma(\pi)$, 
\[
\| \one_{Q g Q} \mid_{\pi^Q} \|_{op} \ll (\log X)^C X^{\frac{\sigma -1}{\sigma}}, \qquad X =  |Q g Q / Q|.
\]

\end{cor}

\begin{proof}
The proof is analogous to Proposition \ref{prop:H-Iwahori-bound}.  The reduction stated in the first paragraph of the proof of Proposition \ref{prop:H-Iwahori-bound} follows from the decomposition appearing in Proposition \ref{prop:H-deep-Iwahori-2} and the coset-size comparison of Lemma \ref{lem:H-degrees}. 

In the third paragraph of the proof of Proposition \ref{prop:H-Iwahori-bound}, the only place where we used the Iwahori assumption is in the fact that the Iwahori operator is $W$-normal. We can replace this with input from Bernstein's uniform admissibility Theorem, Theorem \ref{thm:H-Bernstein}, which implies that any operator of the form $\one_{Q n_a Q}$ is $C$-normal. The rest of the proof remains as in Proposition \ref{prop:H-Iwahori-bound}.
\end{proof}

We are now in a position to prove Theorem \ref{thm:H-General-bound}.

\begin{proof}[Proof of Theorem \ref{thm:H-General-bound}]
Without loss of generality, $P$ is the stabilizer of a finite set $X \subset \mA$, and furthermore, by translating it, let $r$ be large enough such that $X \subseteq \Delta_r$, hence $Q := I_r \subseteq P$.
Then
\[
\|\one_{P S P}\mid_{\pi^P} \|_{op} 
\leq \sum_{PsP \in P \backslash PSP/P} \|\one_{P s P}\mid_{\pi^P} \|_{op}
\leq |P \backslash PSP/P| \max_{PsP \in P \backslash PSP/P} \|\one_{P s P}\mid_{\pi^P} \|_{op},
\]
and combined with Proposition \ref{prop:H-coset-size}, Lemma \ref{lem:H-degrees}, and Corollary \ref{cor:H-deep-Iwahori-bound}, we get 
\[
\ll (\log |S|)^d \max_{s\in S} \|\one_{P s P}\mid_{\pi^P} \|_{op}
\ll (\log |S|)^d \max_{s\in S} \|\one_{Q s Q}\mid_{\pi^Q} \|_{op} 
\ll (\log |S|)^{C+d} |S|^{\frac{\sigma - 1}{\sigma}},
\]
which completes the proof.
\end{proof}

\section{Cutoff} \label{sec:C}

In this section we study cutoff for walks on quotients of buildings, under a very general framework. Recall Definition \ref{defn:I-cutoff}, and let $\widetilde{X}, G, \{T_\ell\},\mathbf{u}$ be as in the formulation of Question \ref{que:cutoff}. The properties we need for our analysis are the following:

\begin{defn} \label{defn:C-growth-trans}
We say that a sequence $\left\{ T_{\ell}\right\} $ of geometric operators on $\widetilde{X}$ is:
\begin{enumerate}
\item \uline{Of \mbox{$\left(d,f\right)$}-exponential growth} if $d>1$ and $f\colon\N\rightarrow\N$ is a monotone subexponential function such that $d^{\ell}/f(\ell) \leq \deg(T_{\ell}) \leq d^{\ell}f(\ell)$;
\item \uline{Transient} if $\left(T_{\ell}\one_{x}\right)(y)$ is globally bounded for $\ell\in\N,x,y\in\widetilde{X}$.
\end{enumerate}
\end{defn}

\begin{example}
These are some known results regarding cutoff in bounded degree:
\begin{enumerate}
\item Ramanujan graphs.

(a) SRW (simple random walk): \cite{LP16} shows that non-bipartite $k$-regular Ramanujan graphs exhibit cutoff w.r.t.\ $T_{\ell}=A^{\ell}$ and $\mathbf{u}=\mathbf{1}$ at $t = \frac{k}{k-2}\log_{k-1}n$, $w=O\left(\sqrt{\log n}\right)$. The result holds in the bipartite case as well, by taking $\mathbf{u}$ to be the space of functions which are constant on even and odd levels in the $k$-regular tree $\widetilde{X}$.

(b) NBRW (nonbacktracking random walk): Can be studied either on directed edges, with $T_{\ell}=B^{\ell}$ where $B$ is the Hashimoto operator, or on vertices, with $T_{\ell}$ summing over vertices of distance $\ell$ in $\widetilde{X}$ \cite{NS23}. Both exhibit cutoff at $t=\log_{k}n$, $w=3\log_{k}\log n$ \cite{LP16}, and under an additional assumption on the girth even $w=O(1)$ \cite{NS23}.

(c) Biregular graphs: analogue results for SRW and NBRW on $(K,k)$-biregular Ramanujan graphs are shown in \cite{EFMP23}.
\item Ramanujan complexes.

(a) Collision-free walks \cite{LLP20}: If $\widetilde{X}$ is a Bruhat-Tits building of a semisimple group, and $T_{\ell}=T^{\ell}$ for a collision-free operator $T$, then Ramanujan quotients of $\widetilde{X}$ exhibit cutoff at $t=\log_{\deg T}n$, $w=O\left(\log\log n\right)$.

(b) SRW on type $\widetilde{A}_{d}$ \cite{CP22}: SRW on the $1$-skeleton of Ramanujan complexes associated with $PGL_d(\F)$ exhibit cutoff at $t=(\left\lceil \frac{d}{2}\right\rceil \left\lfloor \frac{d}{2}\right\rfloor +O(\frac{1}{q_\F}))\log_{q_\F}n$ and $w=O(\sqrt{t})$.
\end{enumerate}
\end{example}

Let $\mB$ be the Bruhat-Tits building of $G$, a connected semisimple group over a non-archimedean local field. Let $\sigma_{0}\in\mB$ be a face and $P = \mathrm{Stab}_{G}(\sigma_{0})$ its stabilizer in $G$. We take $\widetilde{X}$ to be the $G$-orbit $\widetilde{X} = G\sigma_0 \subseteq \mB$. For $G' = G_{\der}$ the derived subgroup of $G$, we call the $G'$-orbits in $\widetilde{X}$ \emph{types}, and take $\mathbf{u}$ to be the space of functions on $\widetilde{X}$ which are constant on each type. Our finite quotients will be $X = \Gamma \backslash \widetilde{X} \subset \Gamma \backslash \mB$, for cocompact lattices $\Gamma \leq G$.

\begin{rem} \label{rem:C-types}
Let $G^0 \leq G$ be as in Definition \ref{defn:BT-G0}, and let $P^0 = P \cap G^0$ the parahoric subgroups corresponding to $\sigma_0$. By Remark \ref{rem:BT-G0=G'I}, $G^0 = G'\cdot P^0$, and the number of types is $[G:G'\cdot P] \leq [G:G'\cdot P^0] = [G:G^0]$, which is finite and independent of the field. We note that $[G:G^0] \leq [\mI^+ : \mI] = $ the number of symmetries in $G$ of a chamber. For instance, $[G:G^0] \leq n$ if $G$ is a form of $PGL_n(F)$, and otherwise $[G:G^0] \leq 8$. Finally, if $G = G^0$ (e.g.\ $SL_{n}(F),Sp_{2n}(F)$), then $G\sigma_{0}$ has a unique type and $\mathbf{u}$ are the constant functions.
\end{rem}

Our first result of this section is cutoff for Ramanujan complexes. We use the definition from \cite{LLP20}, calling a finite quotient $\Gamma \backslash \mB$ \emph{Ramanujan} if $\dim L^2_\sigma(\Gamma \backslash G)^\mI = 0$, for any $2<\sigma <\infty$ (this is stronger than the definition in \cite{LSV05,li2004ramanujan} which  only requires $\dim L^2_\sigma(\Gamma \backslash G)^K = 0$ for a maximal compact $K$). See \cite{LSV05,sarveniazi2007explicit,EP22} for previous examples of Ramanujan complexes of type $\widetilde A$, as well as \cite{DMZ26} for additional types.

\begin{thm} \label{thm:C-cutoff-ram}
If $\{ T_{\ell}\}$ is a sequence of geometric operators on $G\sigma_{0}$ which is transient and of $(d,f)$-exponential growth, then any family of Ramanujan quotients of $\mB$ exhibits cutoff at time $t(n)=\log_{d}n$ w.r.t.\ $T_{\ell}$ and $\mathbf{u}$ (the type functions). The cutoff window $w$ satisfies the following: Let $c_{G}>2(D+W)$, where $D=\dim(\mB)$ and $W$ is the size of the spherical Weyl group of $G$, and let $\log=\log_{d}$. Then $w\colon\N\to\N$ can be any monotone function that satisfies
\begin{equation}
\log f\left(2\log n\right)=o\left(w\left(n\right)\right),\quad c_{G}\log\log n\leq w(n)\qquad\mbox{and}\qquad w\left(n\right)=o\left(\log n\right).\label{eq:window}
\end{equation}
\end{thm}

Let us explain why such $w$ exists and give some examples.

\begin{lem} \label{lem:C-subexp}
Let $f\colon\N\rightarrow\N$ be monotone and subexponential, and interpret $f(x)$ as $f\left(\left\lceil x\right\rceil \right)$ for $x\notin\N$. For any $c>0$ there exists $w\colon\N\to\N$, such that 
\[
\log f\left(2\log n\right)=o\left(w\left(n\right)\right),\quad c\log\log n\leq w(n)\qquad\mbox{and}\qquad w\left(n\right)=o\left(\log n\right).
\]
\end{lem}

\begin{proof}
Subexponentiality of $f$ is equivalent to $\log f\left(2n\right)=o\left(n\right)$, which is equivalent to $\log f\left(2\log n\right)=o\left(\log n\right)$. In addition $c\log\log n=o(\log n)$, so the lemma follows by taking any $w$ which grows faster than $\log f\left(2\log n\right)$ and $c\log\log n$, but slower than $\log n$. 
\end{proof}

\begin{example}
Some examples for $f$, $\deg T_{\ell}$ and $w$ as in Theorem \ref{thm:C-cutoff-ram}:

\noindent\begin{minipage}[t]{1\columnwidth}%
\begin{center}
\begin{tabular}{|c|c|c|}
\hline 
$f\left(\ell\right)$ & $\deg T_{\ell}$ & $w(\ell)$\tabularnewline
\hline 
\hline 
$\Theta\left(1\right)$ & $\Theta(d^{\ell})$ & $c_{G}\log\log\ell$\tabularnewline
\hline 
$\ell^{m}$ & $\vphantom{\overset{.}{X}}\widetilde{\Theta}(d^{\ell})$ & $\max\{m,c_{G}\}\log\log\ell$\tabularnewline
\hline 
$d^{\ell^{\alpha}}$ ($0<\alpha<1$) & $\in\left[d^{\ell-\ell^{\alpha}},d^{\ell+\ell^{\alpha}}\right]$ & $(\log\ell)^{\beta}$ ($\alpha<\beta<1$)\tabularnewline
\hline 
\end{tabular}
\par\end{center}%
\end{minipage}
\end{example}

Next we state a cutoff theorem under the weaker assumption of SXDH and spectral gap.

\begin{defn} \label{defn:C-SXDH-gap}
Let $\mathcal{I}\leq G$ be a Iwahori subgroup. We say that a sequence $\left\{ X_{n}=\Gamma_{n}\backslash\mB\right\} $ of quotients of $\mB$ satisfies the SXDH and spectral gap properties if:
\begin{equation} \label{eq:SXDH}
\begin{aligned}\text{(SXDH)} &  & \exists k\ \forall \tau: & \sum\nolimits_{\sigma\geq\tau}\dim L_\sigma^2(\Gamma_n \backslash G)^\mI \ll (\log|X_n|)^k |X_n|^{2/\tau}\\
\text{(spectral gap)} &  & \exists R>2: & \sum\nolimits_{R\leq\sigma<\infty}\dim L_{\sigma}^{2}(\Gamma_{n}\backslash G)^{\mathcal{I}}=0.
\end{aligned}
\end{equation}
\end{defn}

\begin{rem} \label{rem:C-SXDH}
(1) Let $D$ be the semisimple rank of $G$. The spectral gap property holds when $D \geq 2$ by the explicit form of Kazhdan Property $(T)$ of \cite{Oh02}, and if $D=1$ it holds when $\Gamma_n \leq G$ are congruence subgroups by Property $(\tau)$ \cite{Clo03}.

(2) Clearly Ramanujan complexes satisfy the SXDH. See \cite{EFMP23,EGGG23} for constructions coming from inner forms of $\bfPU_3$ and $\mathbf{SO}_5$ of non-Ramanujan complexes which satisfy the SXDH (which are also almost-transitive, see below).

(3) Using the full power of the Arthur-Ramanujan Conjectures \cite{Arthur1989Unipotent}, and assuming $\Gamma_n \leq G$ are congruence subgroups, we get that the number of $\sigma$ such that $L_{\sigma}^{2}(\Gamma_{n}\backslash G) \ne\{0\}$ is bounded by a constant which depends only on the rank of $G$. 

(4) By the work of Golubev-Kamber \cite[Theorems~1.6~and~1.8]{GK23}, one can replace $\mI$ in the definition of the SXDH with any other compact open group $U$, incurring an additional constant which depends only on $U$. In particular, one may replace $\mI$ with $K$, an extraspecial maximal compact subgroup of $G$.
\end{rem}

We say that a family $\{ X_{n}\}$ of quotients of $\widetilde{X}$ is \emph{almost-transitive} if $\left|\frac{X_{n}}{\mathrm{Aut}(X_{n})}\right|\ll1$. In particular, this holds if $X_{n}$ are all normal covers of some finite quotient of $\widetilde{X}$. A central example is when $\Gamma = \bfG(\mO_F[1/\p])$ is a $\p$-arithmetic lattice in $G = \bfG(F_\p)$, and $X_{n}$ are the principal congruence quotients of $\mB$ by  $\Gamma(\mathfrak{a})=\ker(\Gamma\rightarrow \bfG(\mO_F/\mathfrak{a}))$
for $\mathfrak{a}$ coprime to $\p$.

\begin{thm} \label{thm:C-cutoff-SX}
Let $\{ T_{\ell}\}$ be a sequence of geometric operators on $G\sigma_{0}$ which is transient and of $(d,f)$-exponential growth, and let $\{ X_n = \Gamma_n\backslash\mB\}$ be an almost-transitive family of quotients of $\mB$ satisfying the SXDH and spectral gap properties as in \eqref{eq:SXDH}. Then $\{X_n\}$ exhibits cutoff at time $t(n)=\log_{d} n$ w.r.t.\ $T_{\ell}$ and $\mathbf{u}$ (the type functions). The window $w\colon\N\to\N$ is any monotone function satisfying \eqref{eq:window} for any $c_{G}>R(W+D+k/2+1/2)$.
\end{thm}


We now prove our theorems.

\begin{proof}[Proof of Theorem \ref{thm:C-cutoff-ram}]
Let $X=\Gamma\backslash\mB$, $n=|X|$, $d_{\ell}=\deg T_{\ell}$, and $w$ as in the theorem. We need to show that for any fixed $\varepsilon >0$,
\[
\log n - w(n) \leq t_{1-\varepsilon}(X),\, t_{\varepsilon}(X)\leq\log n + w(n),
\]
holds with respect to $\{ T_{\ell}\}$ and $\mathbf{u}$, for $n$ large enough. Thus, we will often write inequalities which are only true for $n$ large enough, without explicitly stating this (for example, inferring $g(n)\leq h(n)$ from $g(n)=o(h(n))$).

We begin with the lower bound, which holds for any quotient of $\mB$, not necessary Ramanujan. Let $X=\bigsqcup_{i=1}^{r}X^{i}$ be the decomposition of $X$ into types, so that $r=[G:\Gamma G'P] \leq [G:G^0]$ (see Remark \ref{rem:C-types}) and $\mathbf{u}|_{X} = \mathrm{Span}\left\{ \one_{X^{i}}\right\} _{i=1}^{r}$. Denoting $\mathbf{p}^{\ell}=\tfrac{T_{\ell}}{d_{\ell}}\one_{x_{0}}$ (for some $x_{0}\in X$) and $\mathfrak{S}=\mathrm{supp}(\mathbf{p}^{\ell})$, we have $\left|\mathfrak{S}\right|\leq d_{\ell}\leq d^{\ell}f(\ell)$, hence for any $\mathbf{q}\in\mathbf{u}|_{X}^{1}$
\begin{align*}
\left\Vert \mathbf{p}^{\ell}-\mathbf{q}\right\Vert _{TV} &=\max_{A\subseteq X}\left|\mathbf{p}^{\ell}(A)-\mathbf{q}(A)\right|\geq\mathbf{q}(X\backslash\mathfrak{S})=1-\sum_{i=1}^{r}\mathbf{q}(X^{i}\cap\mathfrak{S})\\  
&={\textstyle 1-\sum_{i=1}^{r}\frac{|\mathfrak{S}\cap X_{i}|}{|X_{i}|}\mathbf{q}(X^{i})\geq1-\sum_{i=1}^{r}\frac{d^{\ell}f(\ell)}{n/r}\mathbf{q}(X^{i})=1-\frac{rd^{\ell}f(\ell)}{n}}.
\end{align*}
Therefore, at $\ell=\log n-w\left(n\right)$ we have 
\[
\dist_{TV}\left(\mathbf{p}^{\ell},\mathbf{u}|_{X}^{1}\right)\geq1-\frac{rd^{\ell}f(\ell)}{n}=1-\frac{rf\left(\log n-w(n)\right)}{d^{w(n)}}\geq1-\frac{rf\left(\log n\right)}{d^{w(n)}}\overset{{\scriptscriptstyle n\rightarrow\infty}}{\longrightarrow}1,
\]
since $\log f\left(2\log n\right)=o(w(n))$ implies 
\begin{equation}
\forall C>1\ :\ f(\log n)\leq f(2\log n)=o(C^{w(n)}).\label{eq:f(log(n))}
\end{equation}
In total, we have obtained $t_{1-\varepsilon}(X)\geq\log n-w(n)$ for $n$ large enough.

For the upper bound, we decompose $L^{2}(\Gamma\backslash G)=\bigoplus V_{i}\oplus\bigoplus W_{i}\oplus\bigoplus U_{i}$ as a Hilbert sum of irreducible representation of $G$, where $V_{i}$ are infinite-dimensional and $P$-spherical, $W_{i}$ are one-dimensional and $P$-spherical, and $U_{i}^{P}=0$. The correspondence $X=\Gamma\backslash\widetilde{X}\cong\Gamma\backslash G/P$ then induces 
\begin{equation}
L^{2}\left(X\right)\cong L^{2}\left(\Gamma\backslash G/P\right)\cong L^{2}\left(\Gamma\backslash G\right)^{P}=\bigoplus V_{i}^{P}\oplus\bigoplus W_{i}^{P}.\label{eq:XtoG}
\end{equation}
Since $W_{i}$ are the summands whose structure map factors through $G'$, and $\mathbf{u}|_{X}$ correspond to $\Gamma G'P$-invariant functions on $G$, we have $\mathbf{u}|_{X}=\bigoplus W_{i}^{P}$ and thus $\mathbf{u}|_{X}^{\bot}=\bigoplus V_{i}^{P}$. 

Next, for each $\ell$ choose a set $S_{\ell}\subseteq G$ of size $d_{\ell}$ such that $T_{\ell}(\sigma_{0})=S_{\ell}\sigma_{0}$. It follows from $T_{\ell}$ being geometric that $S_{\ell}$ is $P$-stable, and that the element $\one_{S_{\ell}P}=\one_{PS_{\ell}P}$ in the $P$-Hecke algebra of $G$ acts on $L^{2}\left(\Gamma\backslash G\right)^{P}$ as $T_{\ell}$, via the identification in (\ref{eq:XtoG}). In particular, this implies that $T_{\ell}|_{L^{2}(X)}$ decompose w.r.t.\ $\mathbf{u}|_{X}\oplus\mathbf{u}|_{X}^{\bot}$, and since functions in $\mathbf{u}|_{X}^{\bot}$ have sum zero, we have $\mathbb{P}_{\mathbf{u}|_{X}}\mathbf{p}^{\ell}=\mathbf{p}^{\ell}-\mathbb{P}_{\mathbf{u}|_{X}^{\bot}}\mathbf{p}^{\ell}\in\mathbf{u}|_{X}^{1}$, and note that $T_{\ell}\mathbb{P}_{\mathbf{u}|_{X}^{\bot}}$ is $T_{\ell}$ restricted to $\bigoplus V_{i}^{P}$. Thus,
\begin{equation}\label{eq:upper}
\begin{alignedat}{1}\dist_{TV}\left(\mathbf{p}^{\ell},\mathbf{u}|_{X}^{1}\right) & \leq\dist_{TV}\left(\mathbf{p}^{\ell},\mathbb{P}_{\mathbf{u}|_{X}}\mathbf{p}^{\ell}\right)\leq\sqrt{n}\left\Vert \mathbf{p}^{\ell}-\mathbb{P}_{\mathbf{u}|_{X}}\mathbf{p}^{\ell}\right\Vert _{2}=\tfrac{\sqrt{n}}{d_{\ell}}\left\Vert T_{\ell}\one_{x_{0}}-\mathbb{P}_{\mathbf{u}|_{X}}T_{\ell}\one_{x_{0}}\right\Vert _{2}\\
& =\tfrac{\sqrt{n}}{d_{\ell}}\left\Vert T_{\ell}\one_{x_{0}}-T_{\ell}\mathbb{P}_{\mathbf{u}|_{X}}\one_{x_{0}}\right\Vert _{2}=\tfrac{\sqrt{n}}{d_{\ell}}\left\Vert T_{\ell}\mathbb{P}_{\smash{\mathbf{u}|_{X}^{\bot}}}\one_{x_{0}}\right\Vert _{2}\leq\max_{i}\tfrac{\sqrt{n}}{d_{\ell}}\left\Vert T_{\ell}|_{\smash{V_{i}^{P}}}\right\Vert _{2}.
\end{alignedat}
\end{equation}
By the Ramanujan assumption each $V_{i}$ has rate of decay $\sigma=2$, hence by Theorem \ref{thm:H-General-bound} we have $\left\Vert T_{\ell}|_{\smash{V_{i}^{P}}}\right\Vert _{2}\ll\left(\log d_{\ell}\right)^{W+D}\sqrt{d_{\ell}}$ for every $i$, and altogether
\begin{equation} \label{eq:C-upper-Ram}
\dist_{TV}\left(\mathbf{p}^{\ell},\mathbf{u}|_{X}\right)\ll\left(\log d_{\ell}\right)^{W+D}\left( \frac{n}{d_{\ell}} \right)^{1/2}. 
\end{equation}
We now fix $\ell=\log n+w(n)$, at which 
\begin{equation}
\begin{aligned}\log d_{\ell} & \leq\ell+\log f(\ell)=\log n+w(n)+\log f\left(\log n+w(n)\right)\\ 
& \leq\log n+w(n)+\log f(2\log n)\leq\log n+2w(n)\leq3\log n\ ;\\
\frac{n}{d_{\ell}} & \leq\frac{nf(\ell)}{d^{\ell}}=\frac{f\left(\log n+w(n)\right)}{d^{w(n)}}\leq\frac{f\left(2\log n\right)}{d^{w(n)}}.
\end{aligned}
\label{eq:log_dl}
\end{equation}
Taking $\delta=\frac{W+D}{c_{G}}$ (so $\delta<\frac{1}{2}$), we obtain 
\[
\dist_{TV}\left(\mathbf{p}^{\ell},\mathbf{u}|_{X}\right)\ll\left(3\log n\right)^{W+D}\sqrt{\frac{f(2\log n)}{d^{w(n)}}}=\frac{\left(3\log n\right)^{W+D}}{d^{\delta w(n)}}\sqrt{\frac{f(2\log n)}{d^{(1-2\delta)w(n)}}}.
\]
This goes to zero as $n\rightarrow\infty$: from $(W\!+\!D)\log\log n\leq\delta w(n)$ we obtain $\left(3\log n\right)^{W+D}=O(d^{\delta w(n)})$, and from (\ref{eq:f(log(n))}) we have $f(2\log n)=o(d^{(1-2\delta)w(n)})$. In total, $t_{\varepsilon}(X)\leq\log n+w(n)$ (for $n$ large enough) as claimed. 
\end{proof}

\begin{proof}[Proof of Theorem \ref{thm:C-cutoff-SX}]
Since the proof of the lower bound for $t_{1-\varepsilon}$ was general, we only need to prove the upper bound for $t_{\varepsilon}$, with Ramanujan replaced by SXDH, spectral gap and almost-transitivity, i.e.\  $\frac{n}{|\mathcal{G}|}=O(1)$ for $\mathcal{G}=\Aut(X)$. We begin from (\ref{eq:upper}) without the last step, and then average the starting point $x_{0}$ over its $\mathcal{G}$-translations:
\begin{align*}
\dist_{TV}\left(\mathbf{p}^{\ell},\mathbf{u}|_{X}^{1}\right)^{2} & \leq\tfrac{n}{d_{\ell}^{2}}\left\Vert T_{\ell}\mathbb{P}_{\smash{\mathbf{u}|_{X}^{\bot}}}\one_{x_{0}}\right\Vert _{2}^{2}=\tfrac{n}{d_{\ell}^{2}|\mathcal{G}|}\sum_{g\in\mathcal{G}}\left\Vert T_{\ell}\mathbb{P}_{\mathbf{u}|_{X}^{\bot}}\one_{gx_{0}}\right\Vert _{2}^{2}\\
& \ll\tfrac{1}{d_{\ell}^{2}}\sum_{g\in\mathcal{G}}\left\Vert T_{\ell}\mathbb{P}_{\mathbf{u}|_{X}^{\bot}}\one_{gx_{0}}\right\Vert _{2}^{2}\leq\tfrac{1}{d_{\ell}^{2}}\sum_{x\in X}\left\Vert T_{\ell}\mathbb{P}_{\mathbf{u}|_{X}^{\bot}}\one_{x}\right\Vert _{2}^{2}=\tfrac{1}{d_{\ell}^{2}}\left\Vert T_{\ell}\mathbb{P}_{\mathbf{u}|_{X}^{\bot}}\right\Vert _{F}^{2}
\end{align*}
where the latter is Frobenius norm. Since $T_{\ell}\mathbb{P}_{\mathbf{u}|_{X}^{\bot}}$ is $T_{\ell}$ on $\bigoplus V_{i}^{P}$ and zero on $\bigoplus W_{i}^{P}$,
\[
\tfrac{1}{d_{\ell}^{2}}\left\Vert T_{\ell}\mathbb{P}_{\mathbf{u}|_{X}^{\bot}}\right\Vert _{F}^{2}=\tfrac{1}{d_{\ell}^{2}}\sum_{i}\left\Vert T_{\ell}\big|_{V_{i}^{P}}\right\Vert _{F}^{2}.
\]
If we choose a $T_{\ell}^{*}T_{\ell}$-eigenbasis for each $V_{i}$, and list their eigenvalues together as $\lambda_{1},\ldots,\lambda_{n_{0}}$ ($n_{0}=\dim\mathbf{u}|_{X}^{\bot}=n-\dim\mathbf{u}|_{X}$), then $\left\Vert T_{\ell}\mathbb{P}_{\mathbf{u}|_{X}^{\bot}}\right\Vert _{F}^{2}=\sum_{j=1}^{n_{0}}\lambda_{j}$. For every $1\leq j\leq n_{0}$, let $\sigma_{j}$ be the rate of decay of the $V_{i}$ from which $\lambda_{j}$ came. From Theorem \ref{thm:H-General-bound} we have $\sqrt{\lambda_{j}}\leq\left\Vert T_{\ell}\big|_{V_{i}^{P}}\right\Vert _{2}\leq\left(\log d_{\ell}\right)^{W+D}d_{\ell}^{(\sigma_{j}-1)/\sigma_{j}},$ hence 
\[
\tfrac{1}{d_{\ell}^{2}}\left\Vert T_{\ell}\mathbb{P}_{\mathbf{u}|_{X}^{\bot}}\right\Vert _{F}^{2}=\tfrac{1}{d_{\ell}^{2}}\sum_{j=1}^{n_{0}}\lambda_{j}\leq\left(\log d_{\ell}\right)^{2(W+D)}\sum_{j=1}^{n_{0}}d_{\ell}^{-\frac{2}{\sigma_{j}}}.
\]
We denote by $L_{\sigma}^{2}(\Gamma\backslash G)$ the sum of all $V_{i}$ with matrix coefficients with rate of decay $\sigma$. The SXDH and spectral gap assumptions in \eqref{eq:SXDH} give us 
\begin{align*}
\sum_{j=1}^{n_{0}}d_{\ell}^{-\frac{2}{\sigma_{j}}} & =\sum_{j=1}^{n_{0}}\int_{0}^{d_{\ell}^{-2/R}}\one_{\big[0,d_{\ell}^{-2/\sigma_{j}}\big]}(t)dt=\int_{0}^{d_{\ell}^{-2/R}}\left|\left\{ j:t\leq d_{\ell}^{-\frac{2}{\sigma_{j}}}\right\} \right|dt\\
 & =\int_{0}^{d_{\ell}^{-2/R}}\left|\left\{ j:\sigma_{j}\geq2\log_{1/t}d_{\ell}\right\} \right|dt=\int_{0}^{d_{\ell}^{-2/R}}\sum_{\sigma\geq2\log_{1/t}d_{\ell}}\dim L_{\sigma}^{2}(\Gamma\backslash G)^{P}\,dt\\
 & \leq\int_{0}^{d_{\ell}^{-2/R}}(\log n)^{k}n^{\frac{2}{2\log_{1/t}d_{\ell}}}dt=(\log n)^{k}\int_{0}^{d_{\ell}^{-2/R}}t^{-\log_{d_{\ell}}n}dt.
\end{align*}
We assume $d_{\ell}>n$, since this is the time we are interested in, so that $1-\log_{d_{\ell}}\!n\in\left(0,1\right)$. This gives
\[
\sum_{j=1}^{n_{0}}d_{\ell}^{-\frac{2}{\sigma_{j}}}\leq(\log n)^{k}\int_{0}^{d_{\ell}^{-2/R}}t^{-\log_{d_{\ell}}n}dt=(\log n)^{k}\frac{t^{1-\log_{d_{\ell}}n}}{1-\log_{d_{\ell}}n}\bigg|_{t=0}^{d_{\ell}^{-2/R}}=\frac{(\log n)^{k}(n/d_{\ell})^{(2/R)}}{1-\log_{d_{\ell}}n}.
\]
Altogether we get (compare with \eqref{eq:C-upper-Ram}),
\begin{equation} \label{eq:C-upper-SXDH}
\dist_{TV}\left(\mathbf{p}^{\ell},\mathbf{u}|_{X}\right)^2 \ll \left(\log d_{\ell}\right)^{2(W+D)} \frac{(\log n)^{k}}{1-\log_{d_{\ell}}n} \left( \frac{n}{d_{\ell}} \right)^{2/R}. 
\end{equation}
We now fix $\ell=\log n+w(n),$ at which (\ref{eq:log_dl}) holds, and also
\[
\log_{d_{\ell}}n\leq\frac{\log n}{\log d_{\ell}}\leq\frac{\log n}{\ell-\log f(\ell)}\leq\frac{\log n}{\log n+w(n)-\log f(2\log n)}\leq\frac{\log n}{\log n+w(n)/2}
\]
which yields $\frac{1}{1-\log_{d_{\ell}}n}\leq1+\frac{2\log n}{w(n)}\leq3\log n$. Hence
\[
\left(\log d_{\ell}\right)^{2(W+D)} \frac{(\log n)^{k}}{1-\log_{d_{\ell}}n} \leq \left(3\log n\right)^{2W+2D+k+1}.
\]
Taking $E = 2W+2D+k+1$, $\delta=\frac{E}{c_G}$ (so $\delta<\frac{2}{R}$), and combining everything we get
\[
\dist_{TV}\left(\mathbf{p}^{\ell},\mathbf{u}|_{X}^{1}\right)^{2} \ll (3\log n)^E \left(\frac{f(2\log n)}{d^{w(n)}}\right)^{\tfrac{2}{R}} = \frac{(3\log n)^E}{d^{\delta w(n)}} \left( \frac{f(2\log n)}{d^{(1-\delta R/2)w(n)}}\right)^{\tfrac{2}{R}}.
\]
Now from $\delta c_G \log\log n\leq\delta w(n)$ we obtain $\left(3\log n\right)^E = O(d^{\delta w(n)})$, and from (\ref{eq:f(log(n))}) we have $f(2\log n)=o(d^{(1-\delta R/2)w(n)})$, so that $t_{\varepsilon}(X)\leq\log n+w(n)$ for $n$ large enough.
\end{proof}





\bibliographystyle{amsalpha}
\bibliography{Tbib}

\end{document}